\pdfoutput=1
\documentclass[11pt]{amsart}
\usepackage[utf8]{inputenc}
\usepackage[english]{babel}
\usepackage{cite}
\usepackage{amsthm,amsmath,amssymb,hyperref,amsfonts,tikz,adjustbox}
\usetikzlibrary{cd}
\numberwithin{equation}{section}
\usepackage[shortlabels]{enumitem}
\usepackage[a4paper, total={6in, 8in}]{geometry}
\usepackage{imakeidx}

\newtheorem{theorem}{Theorem}[section]
\newtheorem{example}[theorem]{Example}
\newtheorem{corollary}[theorem]{Corollary}
\newtheorem{lemma}[theorem]{Lemma}
\newtheorem{proposition}[theorem]{Proposition}

\theoremstyle{definition}
\newtheorem{definition}[theorem]{Definition}

\theoremstyle{remark}
\newtheorem{remark}[theorem]{Remark}

\title{(Co)Homology of Partial Smash Products}



\author{Mikhailo Dokuchaev}
\address[Mikhailo Dokuchaev]{Departamento de Matemática, Universidade de São Paulo,
Rua do Matão, 1010, 05508-090 São Paulo, Brazil}
\email{dokucha@ime.usp.br}

\author{Emmanuel Jerez}
\address[Emmanuel Jerez]{Departamento de Matemática, Universidade de São Paulo,
Rua do Matão, 1010, 05508-090 São Paulo, Brazil}
\email{ejerez@ime.usp.br}

\begin{document}

\begin{abstract} 
Given a cocommutative  Hopf algebra $\mathcal{H}$ over a commutative ring $K$  and a symmetric partial action of
$\mathcal{H}$ on a $K$-algebra $A,$ we obtain a first quadrant Grothendieck spectral sequence converging to the Hochschild homology of the smash product $A \# \mathcal{H},$ involving the Hochschild homology of $A$ and the partial homology of 
$\mathcal{H}.$ An analogous third quadrant cohomological spectral sequence is also obtained. The definition of the partial (co)homology of  $\mathcal{H}$ under consideration is based on the category of the partial representations of  $\mathcal{H}.$ A specific partial representation of  $\mathcal{H}$ on a subalgebra   $\mathcal{B}$ of the partial ``Hopf" algebra  $\mathcal{H}_{par} $ is involved in the definition and  we  construct a projective resolution of $\mathcal{B}.$ 
\end{abstract}

 \maketitle

 \renewcommand{\contentsname}{Index}
 
 \section{Introduction}
 
 The notion of a partial group action on a $C^*$-algebra, gradually introduced in \cite{E-1}, \cite{Mc} and \cite{E0}, 
and its successful use (see, in particular,  \cite{E6})
motivated a series of algebraic developments (see the surveys \cite{Ba} and \cite{D3}).  
In particular, a Galois theory based on partial group actions was initiated in \cite{DFP},
which inspired its treatment from the point of view of Galois corings in \cite{CaenDGr} and the definition of a partial (co)action of a Hopf algebra on an algebra in \cite{CaenJan}. In the latter paper related concepts were also considered, duality results were obtained and a partial Hopf-Galois theory was introduced.
This was a starting point for rich and interesting Hopf theoretic developments around partial actions 
\cite{AB2}, \cite{AAB}, \cite{ABDP}, \cite{AB3},
\cite{CasPaqQuaSant}, \cite{ABDP2}, \cite{BatiVerc1}, \cite{AzBaFoFoMa}, \cite{ABV2},  \cite{EBAMMT}, \cite{HuVerc},
\cite{ABCQV},  \cite{AzBaFoFoMa}, \cite{AzMaPaSi}, 
\cite{CasFrePaqQuaTam},
\cite{MartiniPaquesSilva}, \cite{SarVer1}, \cite{SarVer2}, \cite{SarVer3}, \cite{BaHaVe}.

In particular, partial representations of a Hopf algebra $\mathcal{H}$ were introduced in \cite{AB2}, extending the notion of a partial group representation (see \cite{E1} and \cite{DEP}). Moreover,  an algebra   $\mathcal{H}_{par}$ was associated with $\mathcal{H}$ in  \cite{AB2}, called the \textit{partial ``Hopf''} algebra, which has 
 the universal property that each partial representation of 
 $\mathcal{H}$ can be factorized by an algebra morphism from 
  $\mathcal{H}_{par},$ being thus  the Hopf analogue of the partial group algebra (see \cite{DEP}). It was also shown in  \cite{AB2} that there is a partial action of $\mathcal{H}$ on a subalgebra $\mathcal{B}$  of   $\mathcal{H}_{par},$  such that  $\mathcal{H}_{par}$ is isomorphic to the smash product  
  $\mathcal{B} \# \mathcal{H},$ generalizing an earlier result from \cite{DE3},
 established in the case of groups. In addition, if $\mathcal{H}$ possesses an invertible antipode, then  $\mathcal{H}_{par}$ has the structure of a  Hopf algebroid  \cite{AB2}.
  Dual concepts were defined and investigated in \cite{ABCQV}.

In the present  article we study the Hochschild  homology and cohomology of  the partial smash product $A \# \mathcal{H}$ by means of spectral sequences, where $\mathcal{H}$ is a cocommutative Hopf algebra, whose partial action on the algebra $A$ is symmetric (see Section~\ref{sec:HopfParRepParAc} for definitions). Earlier,  in \cite{AlAlRePartialCohomology},  group cohomology based on partial representations was introduced and a Grothendieck
spectral sequence was produced, relating the Hochschild  cohomology of the skew group ring
 $A \rtimes G$ by a unital partial action of $G$ on an algebra $A$ with the Hochschild  cohomology of $A$ and the partial group cohomology of $G.$ This was extended in \cite{MDEJ} to the crossed product $A \ast G$ by  a unital twisted partial action of $G$ on $A,$ whose twist takes values in the base field, establishing also a similar result for the Hochschild  homology. The treatment in 
  \cite{MDEJ} was based on the theory of partial projective group representations and the related  novel concept of a twisted  partial group algebra. 
  
  We begin by giving some preliminaries around partial actions and partial representations of Hopf algebras and Hochschild  (co)homology in Section~\ref{sec:prelim}. Section~\ref{sec:homology} is dedicated to the  Hochschild  homology of  the smash product 
  $A \# \mathcal{H},$ where   $\mathcal{H}$ is a cocommutative Hopf algebra over a commutative ring $K$, whose   partial action on a unital algebra $A$ is symmetric. For the main result we also assume that $\mathcal{H}$ is projective over $K.$ The idea is to use Grothendieck's Theorem \cite[Theorem 10.48]{RotmanAnInToHoAl} to obtain a first quadrant spectral sequence $E^r$ converging to the Hochschild  homology of $A \#\mathcal{H}$ with values in a $A \#\mathcal{H}$-bimodule $M,$ and such that 
 $$E^2_{p,q} = \operatorname{Tor}_p^{\mathcal{H}_{par}}(\mathcal{B},H_q(A,M) ),$$ where $\mathcal{B}$ is the above mentioned subalgebra of $\mathcal{H}_{par}.$ 
 
 In order to prepare the ingredients for the use of Grothendieck's Theorem, we work with  the right exact functors of the form 
 $$ F(-):=A\# \mathcal{H} \otimes_{(A\# \mathcal{H})^e}-, \;\;\; F_1(-):= A \otimes_{A^e}- \;\;\;\text{ and }  \;\;\; F_2(-):= \mathcal{B} \otimes_{\mathcal{H}_{par}}-,$$ where
 $A^e$ stands for the enveloping algebra $A \otimes A^{op}$ of $A,$ with 
 $A^{op}$ denoting the opposite algebra of $A$, and the meaning of 
 $(A\# \mathcal{H})^e$ is similar. Observe that the left derived functors of $F_1$ and $F$ compute the Hochschild homology of $A$ and $A\# \mathcal{H},$ respectively. One of the main steps is to show that the functors $F_2F_1$ and $F$ are naturally isomorphic,
 when applied to $A\# \mathcal{H}$-bimodules.  It is obtained in Corollary~\ref{c_F2F1F} as a consequence of a more general fact, Proposition~\ref{p_inXB}, which states that 
     the bifunctors 
      $$   -\otimes_{\mathcal{H}_{par}} (A \otimes_{A^e}-)
     \;\;\;\; \text{ and } \;\;\;\;
     (- \otimes_{\mathcal{B}} A \# \mathcal{H}) \otimes_{(A \# \mathcal{H})^e}-,$$
 defined on $\operatorname{Mod-}\!\mathcal{H}_{par} \times (A \# \mathcal{H})^e\!\operatorname{-Mod} ,$    are naturally isomorphic. In fact, Proposition~\ref{p_inXB} is a crucial technical  tool, which is also used to make the second main step towards the use of Grothendieck's Theorem, namely Propsition~\ref{p_F1sendPtoF2acylic}, which says that 
     $F_1$ sends projective $A \# \mathcal{H}$-bimodules to left $F_2$-acyclic mo\-du\-les.
     
      A considerable  point is to show in Lemma~\ref{l_AskewP} that if $\mathcal{H}$ is projective over $K,$ then $(A \# \mathcal{H})^e$ is  projective over $A^e.$ Consequently, any projective  $(A \# \mathcal{H})$-bimodule is a projective $A$-bimodule, so that applying the left derived functor of $F_1$  to $M$ via taking a projective resolution of $M$ in the category of 
       $A \# \mathcal{H}$-bimodules computes the usual Hochschild homology of $A$ with values in the $A$-bimodule $M$ 
      (see   Remark~\ref{r_projectiveAskewGmodulesAreAprojectives}).
      
       These facts are used to obtain 
    the main result of the section, 
     Theorem~\ref{t_SSHH}, which states the existence of the above mentioned spectral sequence. As an application, if the algebra $A$ under the symmetric partial action of the cocommutative Hopf algebra $\mathcal{H}$ is separable, then we obtain an isomorphism 
    $$H_{n}(A \#\mathcal{H}, M) \cong \operatorname{Tor}_n^{\mathcal{H}_{par}}(\mathcal{B}, M/[A,M]),$$ 
    (see Example~\ref{ex:separable}). Furthermore, 
 if  $\mathcal{H}$ is the group algebra $KG$ of a group $G,$ then the spectral sequence of Theorem~\ref{t_SSHC} takes the form
    \[
         E^2_{p,q} = H_p^{par}(G,H_q(A,M) ) \Rightarrow H_{p+q}(A \rtimes G, M), 
     \]
     where $H^{par}_\bullet(G,-):= \operatorname{Tor}_\bullet^{K_{par}G}(B, -),$ the partial homology  of $G$ introduced in \cite{MMAMDDHKPartialHomology}
(see Example~\ref{ex:group}).

A dual work is done in  Section~\ref{sec:cohomology} to deal with cohomology. The main functors under consideration are  of the form
$$G_1:= \operatorname{Hom}_{A^e}(A,-), \;\;\;
    G_2:= \operatorname{Hom}_{\mathcal{H}_{par}}(\mathcal{B},-)
    \;\;\; \text{and} \;\;\; 
    G:= \operatorname{Hom}_{(A \# \mathcal{H})^e}(A \# \mathcal{H},-).$$ Note that $G_1$ is used to compute the Hochschild cohomology of  $A$, whereas for the co\-ho\-mo\-lo\-gy of  $A \# \mathcal{H}$ the functor $G$ is employed. These functors are used to apply a variation of the  Grothendieck spectral sequence \cite[Theorem 10.47]{RotmanAnInToHoAl}. The crucial steps are  Corollary~\ref{c_G2G1congG}, stating that the functor $G_2G_1$ and $G$ are naturally isomorphic, and Proposition~\ref{prop:G1sendItoD2acylic}, which says that  $G_1$ sends injective $A \# \mathcal{H}$-bimodules to right $G_2$-acyclic modules. Both of them are obtained applying an important technical tool, Proposition~\ref{p_bifunctorscohomologyisomorphism}, which is a result dual to the above mentioned  Proposition~\ref{p_inXB} on bifunctor isomorphism. The main fact is Theorem~\ref{t_SSHC}, which asserts the existence of a third quadrant cohomological spectral sequence $E_r$ such that
    \[
        E_2^{p,q} = \operatorname{Ext}^p_{\mathcal{H}_{par}}(\mathcal{B},H^q(A,M) ) \Rightarrow H^{p+q}(A \#\mathcal{H}, M), 
    \]
    where $\mathcal{H}$ is a cocommutative Hopf $K$-algebra, projective as a $K$-module, whose partial action on $A$ is  symmetric, and   $M$ is an arbitrary  $A \# \mathcal{H}$-bimodule.
    Dually to the homological case, in  the proof of  Theorem~\ref{t_SSHC} we use that any injective resolution of the 
    $A \# \mathcal{H}$-bimodule $M$ is also an injective resolution of $M$ as an $A$-bimodule (see Remark~\ref{r_Hochschild_cohomology}).
    
    In Section~\ref{sec:Hopf(co)homology}, for a 
   cocommutative Hopf algebra $\mathcal{H}$ and a  left $\mathcal{H}_{par}$-modue $M$, we define the partial Hopf homology and cohomology of $\mathcal{H}$ with coefficients in $M$ by
  $$     H^{par}_\bullet(\mathcal{H}, M) := \operatorname{Tor}_\bullet^{\mathcal{H}_{par}}(\mathcal{B}, M) \;\;\;\;\text{and}\\\\\;\;\;\;
        H_{par}^\bullet(\mathcal{H}, M) := \operatorname{Ext}^\bullet_{\mathcal{H}_{par}}(\mathcal{B}, M),$$  respectively.
        Then the above mentioned spectral sequences take the following forms (see Theorem~\ref{t_SS}):
   
    \[
        E^2_{p,q} = H^{par}_p(\mathcal{H},H_q(A,M) ) \Rightarrow H_{p+q}(A \#\mathcal{H}, M),
    \]
    and
    \[
        E_2^{p,q} = H^p_{par}(\mathcal{H},H^q(A,M) ) \Rightarrow H^{p+q}(A \#\mathcal{H}, M). 
    \]
    Note that the latter sequence extends for the Hopf theoretic setting the third quadrant cohomological spectral sequence obtained for the case of unital partial group actions in \cite[Theorem 4.1]{AlAlRePartialCohomology}. It is also shown in Section~\ref{sec:Hopf(co)homology} that if the action of $\mathcal{H}$ on $A$ is global, then we obtain the following global versions of our spectral sequences (see Corollary~\ref{cor:globalHopf}):
 \[
     E^2_{p,q} = \operatorname{Tor}_p^{\mathcal{H}}(K,H_q(A,M) ) \Rightarrow H_{p+q}(A \#\mathcal{H}, M),
 \]
 and
 \[
         E_2^{p,q} = \operatorname{Ext}^p_{\mathcal{H}}(K,H^q(A,M) ) \Rightarrow H^{p+q}(A \#\mathcal{H}, M).
 \]

 In the final Section~\ref{sec:projResol} we construct a projective resolution of $\mathcal{B}$ (see  
 Proposition~\ref{prop:projResol}) by means of a simplicial module which gives rise to an acyclic complex. Note that a projective resolution of $\mathcal{B}$ for the case of groups was obtained in
 \cite{DKhS2}. 
% \bibitem{DKhS2} M.\ Dokuchaev, M.\ Khrypchenko, J.\ J.\ Sim\'{o}n, Globalization of group cohomology in the sense 
%of Alvares-Alves-Redondo, {\it J. Algebra}, {\bf  546},  (2020),  604--640.

 \underline{In all what follows} $K$ will stand for a commutative (associative) unital ring.

\section{Preliminaries}\label{sec:prelim}

In this section, for the reader's convenience, we recall some facts on partial Hopf representations, partial Hopf actions, and Hochschild (co)homology.

\subsection{Partial representations and partial actions of Hopf algebras}\label{sec:HopfParRepParAc}

%The following definitions and results were extracted from %\cite{AB2}.

 The first definition of the concept of a partial representation of a Hopf algebra was given in \cite{AB2} in an asymmetric way, justified by the definition of a partial group action with one-sided ideals considered in \cite{CaenJan} and some constructions. Nevertheless, it became clear that a symmetric definition  gives additional advantages, and the authors of \cite{ALVES2015137} introduced the next:

\begin{definition} \cite[Definition 3.1]{ALVES2015137}
    Let $\mathcal{H}$ be a Hopf $K$-algebra, and let $A$ be a unital $K$-algebra. A \textit{partial representation} of $\mathcal{H}$ in $A$ is a linear map $\pi: \mathcal{H} \to A$ such that
    \begin{enumerate}[(PR1)]
        \item $\pi(1_{\mathcal{H}})=1_A$,
        \item $\pi(h)\pi(k_{(1)})\pi(S(k_{(2)})) = \pi(hk_{(1)})\pi(S(k_{(2)}))$,
        \item $\pi(h_{(1)})\pi(S(h_{(2)}))\pi(k) = \pi(h_{(1)})\pi(S(h_{(2)})k)$,
        \item $\pi(h)\pi(S(k_{(1)}))\pi(k_{(2)}) = \pi(hS(k_{(1)}))\pi(k_{(2)})$,
        \item $\pi(S(h_{(1)}))\pi(h_{(2)})\pi(k) = \pi(S(h_{(1)}))\pi(h_{(2)}k)$,
    \end{enumerate} 
    for all $h,k \in H.$
    \end{definition}
    
    As in the case of partial group representations, morphisms of partial representations of a fixed Hopf algebra $\mathcal{H}$ are defined in most natural way: if the pair $ \pi:\mathcal{H} \to A$ and $\pi': \mathcal{H} \to A' $ are partial representations, then by a \textit{morphism} $\pi \to \pi ' $ we understand an algebra homomorphism $f: A \to A'$ such that $\pi' = f \circ \pi$. Following \cite{ALVES2015137} we denote by $\operatorname{ParRep}_{\mathcal{H}}$  the category of the partial representations of $\mathcal{H}$ and their morphisms.

\begin{remark}
   As it is pointed out in   \cite[Remark 3.2]{ALVES2015137}, if the Hopf algebra $\mathcal{H}$ is cocommutative, then the items (PR3) and (PR4) can be removed from the definition of a partial representation, being consequences of (PR1), (PR2) and (PR5). 
\end{remark}

 Another crucial for us concept is that of a partial Hopf action first defined in \cite{CaenJan}.

\begin{definition} (see \cite{AB2}).
    A left partial action of a Hopf algebra $\mathcal{H}$ on a unital algebra $A$ is a linear map
    \begin{align*}
        \cdot : \mathcal{H} \otimes A & \to A \\
                h \otimes a &\mapsto h \cdot a,
    \end{align*}
    such that
    \begin{itemize}
        \item [(PA1)] $1_{\mathcal{H}} \cdot a = a$ for all $a \in A$;
        \item [(PA2)] $h \cdot (ab) = (h_{(1)} \cdot a)(h_{(2)} \cdot b)$, for all $h \in \mathcal{H}$, $a, b \in A$;
        \item [(PA3)] $h \cdot (k \cdot a) = (h_{(1)} \cdot 1_A)(h_{(2)}k \cdot a)$ for all $h,k \in \mathcal{H}$, $a \in A$. 
    \end{itemize}
    The algebra $A$, on which $\mathcal{H}$ acts partially is called a partial left $\mathcal{H}$-module algebra. Recall also that  a partial action of $\mathcal{H}$ on $A$ is said to be \textit{symmetric} if in addition it satisfies
    \begin{enumerate}[(PA4)]
        \item $h \cdot (k \cdot a) = (h_{(1)}k \cdot a)(h_{(2)} \cdot 1_A)$ for all $h,k \in \mathcal{H}$, $a \in A$. 
    \end{enumerate}
\end{definition}

Given a partial action of a Hopf algebra $\mathcal{H}$ on a unital algebra $A$, an associative product on $A \otimes \mathcal{H}$ is defined  by
\[
    (a \otimes h)(b \otimes k) = a(h_{(1)} \cdot b) \otimes h_{(2)}k.
\]
Then the \textit{partial smash product} (see \cite{CaenJan}) is the unital algebra
\[
    A \# \mathcal{H} := (A \otimes \mathcal{H})(1_A \otimes 1_{\mathcal{H}}):= \{ x (1_A \otimes 1_{\mathcal{H}} ) : x \in A \otimes \mathcal{H} \}.
\] 
 This algebra is generated by the  elements of the form
\[
    a \# h = a(h_{(1)} \cdot 1_A) \otimes h_{(2)}.
\]

\begin{lemma}\label{lemma from Alvares-Batista} (see \cite{AB2}).
    Let $\mathcal{H}$ be a Hopf algebra acting partially on a unital algebra $A.$ Then in the partial smash product $A \# \mathcal{H}$ we have:
    \begin{enumerate}[(i)]
        \item $(a \# h)(b \# k) = a(h_{(1)} \cdot b) \# h_{(2)} k$;
        \item $a \# h = a(h_{(1)} \cdot 1_A) \# h_{(2)}$;
        \item the map $\phi_0: A \to A \# \mathcal{H}$ given by $\phi_0(a)=a \# 1_{\mathcal{H}}$ is an algebra homomorphism.
    \end{enumerate}
\end{lemma}

It is convenient for us to single out the following fact:
\begin{lemma} \label{l_commutSmash}
Suppose that the partial action of a Hopf algebra $\mathcal{H}$ on a unital algebra $A$ is symmetric. Then
     \[
         (b \# 1_{\mathcal{H}})(1_A \# S(h))=(1_A \# S(h_{(1)}))(h_{(2)}\cdot b \# 1_{\mathcal{H}}),
     \] for all  $h \in \mathcal{H}$ and $b \in A.$
 \end{lemma}
 
 \begin{proof}  
Since our partial action is symmetric, in view of Lemma \ref{lemma from Alvares-Batista}  we have:
 \begin{align*} (1_A \# S(h_{(1)}))(h_{(2)}\cdot b \# 1_{\mathcal{H}})&=  S(h_{(2)})(h_{(3)}\cdot b) \#  S(h_{(1)})\\&=
  (S(h_{(3)}) h_{(4)})\cdot b) (S(h_{(2)}) \cdot 1_A) \#  S(h_{(1)})\\&= 
  (\varepsilon (h_{(3)})1_{\mathcal{H}} )\cdot b) (S(h_{(2)}) \cdot 1_A) \#  S(h_{(1)})\\&= 
   b (S(h_{(2)}) \cdot 1_A) \#  S(h_{(1)})\\&= 
    b \# S(h)= (b \# 1_{\mathcal{H}})(1_A \# S(h)).
   \end{align*} \end{proof}
   
    %  \begin{align*}
   %      (b \# 1_{\mathcal{H}})(1_A \# S(h)) &= (b \# 1_{\mathcal{H}})(S(h_{(1)})\cdot1_A \# S(h_{(2)})) \\ 
        % &= (b \# 1_{\mathcal{H}})(S(\varepsilon(h_{(1)})h_{(2)})\cdot1_A \# S(h_{(3)})) \\
        % &= (\varepsilon(h_{(1)})1_{\mathcal{H}}\cdot b \# 1_{\mathcal{H}})(S(h_{(2)})\cdot1_A \# S(h_{(3)})) \\
        % &= (S(h_{(1)})h_{(2)}\cdot b \# 1_{\mathcal{H}})(S(h_{(3)})\cdot1_A \# S(h_{(4)})) \\
        % &= (S(h_{(1)})h_{(2)}\cdot b) (S(h_{(3)})\cdot 1_A) \# S(h_{(4)}) \\
        % &= S(h_{(1)}) \cdot (h_{(2)}\cdot b) \# S(h_{(3)}) \\
        % &= (1_A\# S(h_{(1)}) ) (h_{(2)}\cdot b \# 1_{\mathcal{H}}).
     %\end{align*}
 %\end{proof}

We proceed by recalling the definition of the associative algebra 
$\mathcal{H}_{par}$ which governs the partial  representations of a Hopf algebra $H.$

\begin{definition} \cite[Definition 4.1]{ALVES2015137}
    Let $\mathcal{H}$ be a Hopf algebra and let $T(\mathcal{H})$ be the tensor algebra of the $K$-module $\mathcal{H}$. The \textit{partial ``Hopf''} algebra $\mathcal{H}_{par}$ is the quotient of $T(\mathcal{H})$ by the ideal $I$ generated by the elements of the form
    \begin{enumerate}[(1)]
        \item $1_{\mathcal{H}} - 1_{T(\mathcal{H})}$;
        \item $h \otimes k_{(1)} \otimes S(k_{(2)}) - h k_{(1)} \otimes S(k_{(2)})$, for all $h,k \in \mathcal{H}$;
        \item $h_{(1)} \otimes S(h_{(2)}) \otimes k - h_{(1)} \otimes S(h_{(2)}) k$, for all $h,k \in \mathcal{H}$;
        \item $h \otimes S(k_{(1)}) \otimes k_{(2)} - h S(k_{(1)}) \otimes k_{(2)}$, for all $h,k \in \mathcal{H}$;
        \item $S(h_{(1)}) \otimes h_{(2)} \otimes k - S(h_{(1)}) \otimes h_{(2)} k $, for all $h,k \in \mathcal{H}$.
    \end{enumerate}
\end{definition}

Observe that by \cite[Theorem 4.10]{ALVES2015137} the algebra  $\mathcal{H}_{par}$ possesses the structure of a Hopf algebroid. Denote by $[h]$ the class of $h \in \mathcal{H}$ in 
$\mathcal{H}_{par}$ and consider the map 
\begin{align*}
    [_\_]: \mathcal{H}   & \to \mathcal{H}_{par} \\ 
                h       & \mapsto [h].
\end{align*}  
Recall from \cite{ALVES2015137} the following easily verified relations: 
\begin{enumerate}
    \item $[\alpha h + \beta k] = \alpha[h] + \beta[k]$, for all $\alpha, \beta \in K$ and $h,k \in \mathcal{H}$;
    \item $[1_{\mathcal{H}}] = 1_{\mathcal{H}_{par}}$;
    \item $[h]  [k_{(1)}] [S(k_{(2)})] = [h k_{(1)}] [S(k_{(2)})]$, for all $h,k \in \mathcal{H}$;        
    \item $[h_{(1)} ] [S(h_{(2)})] [k] = [h_{(1)}] [S(h_{(2)}) k]$, for all $h,k \in \mathcal{H}$;
    \item $[h] [S(k_{(1)})] [k_{(2)}] = [h S(k_{(1)})] [k_{(2)}]$, for all $h,k \in \mathcal{H}$;
    \item $[S(h_{(1)})]  [h_{(2)}] [k] = [S(h_{(1)})] [h_{(2)} k] $, for all $h,k \in \mathcal{H}$.
\end{enumerate}

 Thus, the linear map $[_\_]$  is a partial representation  $\mathcal{H} \to \mathcal{H}_{par}$. The following  universal property of the partial ``Hopf'' algebra  expresses the control of $\mathcal{H}_{par}$ on the partial representations of  $H.$

\begin{theorem}(see \cite[Theorem 4.2]{ALVES2015137}) \label{t_Universal Property Partial Actions}
    For every partial representation $\pi: \mathcal{H} \to A$ there is a unique morphism of algebras $\hat{\pi}: \mathcal{H}_{par} \to A$ such that $\pi= \hat{\pi} \circ [_\_]$. Conversely, given an algebra morphism $\pi: \mathcal{H}_{par} \to A,$ there exists a unique partial representation $\pi_\phi: \mathcal{H} \to A$ such that $\phi = \hat{\pi}_\phi$.
\end{theorem}

 The universal property of $\mathcal{H}_{par}$ relates the modules over $\mathcal{H}_{par}$ with the partial $H$-modules, the latter being defined as follows:

\begin{definition}\cite[Definition 5.1]{ALVES2015137}
    Let $\mathcal{H}$ be a Hopf algebra. A \textit{partial module} over $\mathcal{H}$ is a pair $(M, \pi)$, where $M$ is a $K$-module and $\pi: \mathcal{H} \to \operatorname{End}_K(M)$ is a left partial representation of $\mathcal{H}$.
    \end{definition}

  As in \cite{ALVES2015137}, by a \textit{morphism}   $(M, \pi) \to (M', \pi')$ of partial $\mathcal{H}$-modules we mean    a $K$-linear map $f: M \to M'$ such that $f \circ \pi(h) = \pi'(h) \circ f$ for all $h \in \mathcal{H},$ and by $_\mathcal{H} \mathcal{M}^{par}$ we denote  the category of the left partial $\mathcal{H}$-modules and their  morphisms.

The following fact is
 \cite[Corollary 5.3]{ALVES2015137}.
\begin{proposition}\label{prop:CorollaryFromAlves-Batista-Vercruyse}
    There is an isomorphism of categories $_\mathcal{H}\mathcal{M}^{par} \cong \mathcal{H}_{par}$-\textbf{Mod}. Given a  partial $\mathcal{H}$-module $(M, \pi),$ the action of $\mathcal{H}_{par}$ on $M$ is determined by
    \begin{equation}
        [h] \triangleright m := \pi_h(m).
    \end{equation}
\end{proposition}

\subsection{Some background on Hochschild (co)homology}

For the reader's convenience, we proceed by recalling some well-known background on Hochschild (co)homology.

\begin{definition}
    If $A$ is a $K$-algebra, where $K$ is a unital commutative ring, then its \textit{enveloping algebra} is
    \[
        A^e = A \otimes_K A^{op},
    \] 
    where $A^{op}$ stands for the opposite algebra of $A$.
\end{definition}

\begin{proposition}
    Let $X$ be an $A$-bimodule, then $X$ is a left $A^e$-module with action
    \[
        (a\otimes b)\cdot x := a \cdot x \cdot b
    \]
    and is a right $A^e$-module with action
    \[
        x \cdot (a\otimes b) := b \cdot x \cdot a.
    \]
\end{proposition}

\begin{definition}
    Let $A$ be a unital $K$-algebra and $M$ an $A$-bimodule.  The Hochschild homology of $A$ with coefficients in $M$ is defined by
    \[
        H_\bullet(A,M):= \operatorname{Tor}_\bullet^{A^e}(A,M),
    \]
    i.e., $H_\bullet(A,-)$ is the left derived functor of $A \otimes_{A^e}-$. Dually, we define the Hochschild cohomology of $A$ with coefficients in $M$ by
    \[
        H^\bullet(A,M):= \operatorname{Ext}^\bullet_{A^e}(A,M).
    \]
    
\end{definition}

The following easy property of tensor products of bimodules will be used constantly in the development of this work.

\begin{lemma} \label{l_envl}
    Let $X$ and $Y$ be $A$-bimodules. Then for every $a \in A$, $x \in X$ and $y \in Y$ we have
    \[
        a \cdot x \otimes_{A^e} y = x \otimes_{A^e} y \cdot a
    \]
    and
    \[
        x \cdot a\otimes_{A^e} y = x \otimes_{A^e} a \cdot y.
    \]
\end{lemma}

\begin{proof} 
    By direct computations we obtain

    \begin{minipage}{.45\textwidth}
        \begin{align*}
            a \cdot x \otimes_{A^e} y   & = x \cdot (1_A \otimes a) \otimes_{A^e} y \\ 
                                        & = x \otimes_{A^e} (1_A \otimes a) \cdot y \\ 
                                        & = x \otimes_{A^e} y \cdot a,
        \end{align*}
    \end{minipage}%
    \begin{minipage}{0.45\textwidth}
        \begin{align*}
            x \cdot a  \otimes_{A^e} y  & = x \cdot (a \otimes 1_A) \otimes_{A^e} y \\ 
                                        & = x \otimes_{A^e} (a \otimes 1_A) \cdot y \\ 
                                        & = x \otimes_{A^e} a \cdot y.
        \end{align*}
    \end{minipage}

\end{proof}

%%%%%%%

\section{Homology of the partial smash product}\label{sec:homology}

\underline{In all what follows} $\mathcal{H}$ will be a cocommutative Hopf $K$-algebra, $A$ a unital $K$-algebra and 
 \begin{align*}
     \cdot : \mathcal{H} \otimes A &\to A \\
             h \otimes a &\mapsto h \cdot a
 \end{align*}
 a symmetric partial action of $\mathcal{H}$ on $A$.
 
 \begin{proposition} \label{p_parrep}
     Let  $R$ be a $K$-algebra, $\pi: \mathcal{H} \to R$ a partial representation and $X$ an $R$-bimodule. Then the map $\pi': \mathcal{H} \to \operatorname{End}_K(X),$ given by
     \[
         \pi'_h(x):= \sum \pi(h_1) \cdot x \cdot \pi(S(h_2)),
     \]
     is a partial representation.
 \end{proposition}
 
 \begin{proof} \,
     \begin{enumerate}[(PR1)]
         \item $ \pi'_{1}(x) =   \pi(1) \cdot x \cdot \pi(S(1)) = \pi(1) \cdot x \cdot \pi(1) = x.$
         \item \addtocounter{enumi}{2}
         \begin{align*}
             \pi'_h  \pi'_{k_{(1)}} \pi'_{S(k_{(2)})}(x) &= \pi'_h \pi'_{k_{(1)}} \Big( \pi(S(k_{(2)})) \cdot x \cdot \pi(S^2(k_{(3)})) \Big) \\
             &= \pi'_h \Big( \pi(k_{(1)}) \pi(S(k_{(3)})) \cdot x \cdot \pi(S^2(k_{(4)})) \pi(S(k_{(2)})) \Big) \\ 
             &= \pi(h_{(1)}) \pi(k_{(1)}) \pi(S(k_{(3)})) \cdot x \cdot \pi(S^2(k_{(4)})) \pi(S(k_{(2)})) \pi(S(h_{(2)})) \\
          (\text{By the cocommutativity})   &= \pi(h_{(1)}k_{(1)}) \pi(S(k_{(3)})) \cdot x \cdot \pi(S^2(k_{(4)})) \pi(S(h_{(2)}k_{(2)})) \\
             &= \pi'_{hk_{(1)}} \pi'_{S(k_{(2)})}(x).
         \end{align*}
         
         \item 
         \begin{align*}
             \pi'_{S(h_{(1)})}\pi'_{h_{(2)}} \pi'_{k}(x) 
             &= \pi(S(h_{(1)}))\pi(h_{(3)})\pi(k_{(1)}) \cdot x \cdot \pi(S(k_{(2)}))\pi(S(h_{(4)}))\pi(S^2(h_{(2)})) \\ 
             (\text{By the cocommutativity}) &= \pi(S(h_{(1)}))\pi(h_{(3)}k_{(1)}) \cdot x \cdot \pi(S(h_{(4)}k_{(2)})) \pi(S^2(h_{(2)})) \\
             &= \pi'_{S(h_{(1)})}\pi'_{h_{(2)}k}(x).
         \end{align*}
     \end{enumerate}
 \end{proof}
 
 %It is well known that the map
 %\begin{align*}
  %   \phi_0:A    &\to A \# \mathcal{H} \\ 
   %         a    &\mapsto a \#1_{\mathcal{H}}
 %\end{align*}
 %is a morphism of $K$-algebras (see \cite{AB2}).
 
 Let $M$ be an $A \#\mathcal{H}$-bimodule. Then, $M$ is an $A$-bimodule with the actions induced by the map $\phi_0$ from Lemma~\ref{lemma from Alvares-Batista}, i.e., the $A$-bimodule structure  is defined by 
 \begin{equation} \label{e_AbiModSt}
     a \cdot m := (a \#1_{\mathcal{H}}) \cdot m \text{ and } m \cdot a := m \cdot (a \#1_{\mathcal{H}}).
 \end{equation}
 
 Since our partial action is symmetric, by \cite[Example 3.7]{ALVES2015137}, the map $\pi_0: \mathcal{H} \to A \# \mathcal{H}$, given by $h \mapsto 1_A \# h$, is a partial representation of $\mathcal{H} $ into the partial smash product $A \# \mathcal{H}$. Consequently, by Proposition \ref{p_parrep} the map $\pi':\mathcal{H} \to \operatorname{End}_K(M)$, such that
 \[
    \pi'_h(m):=(1_A \# h_{(1)})\cdot m \cdot (1_A \# S(h_{(2)})),
 \]
  is a partial representation of $\mathcal{H}$. Thus, by Proposition~\ref{prop:CorollaryFromAlves-Batista-Vercruyse}, $M$ is an $\mathcal{H}_{par}$-module with action given by
     \begin{equation} \label{e_Hpar-modM}
         [h] \triangleright  m := \pi_h'(m) = (1_A \# h_{(1)})\cdot m \cdot (1_A \# S(h_{(2)})).
     \end{equation}
 
 Notice that given a partial representation $\phi:\mathcal{H} \to \mathcal{R}$ into a unital algebra $\mathcal{R}$, we have that $\mathcal{R}$ becomes a left $\mathcal{H}_{par}$-module  by setting
 \begin{equation}
     [h] \triangleright r := \phi(h)r,
 \end{equation}
 in view of the isomorphism $\mathcal{R} \cong \operatorname{End}_{\mathcal{R}}(\mathcal{R}).$
 
 \begin{proposition} \label{p_AskewGmIsKparGm}
     Let $f:X \to Y$ be a map of $A \# \mathcal{H}$-bimodules. Then, $f$ is a morphism of $\mathcal{H}_{par}$-modules.
 \end{proposition}
 
 \begin{proof}
     By direct computation we obtain
     \[
         f([h] \triangleright x) = f((1_A \# h_{(1)})\cdot x \cdot (1_A \# S(h_{(2)}))) = (1_A \# h_{(1)})\cdot f(x) \cdot (1_A \# S(h_{(2)})) = [h] \triangleright f(x),
     \]
     for all $x \in X$ and $h \in \mathcal{H}$.
 \end{proof}
 
 Recall that since our  partial action $\cdot: \mathcal{H} \otimes A  \to A$ is symmetric, then by \cite[Example 3.5]{ALVES2015137}  the map $h \mapsto \lambda_h \in \operatorname{End}_K(A)$, where $\lambda_h(a):= h \cdot a,$ is a partial representation of $\mathcal{H}$ on $A$. Consequently, $A$ is a $\mathcal{H}_{par}$-module with action
 \begin{equation} \label{e_Hpar-modA}
     [h] \triangleright a := h \cdot a.
 \end{equation}

 For each $h \in c$ define
 \[
     e_h:=[h_{(1)}][S(h_{(2)})], 
 \]
 explicitly we have that
 \[
     e_h = \mu \circ ([\_] \otimes [\_]) \circ (1 \otimes S) \circ \Delta (h), 
		%\text{ and } \tilde{e}_h = \mu \circ ([\_] \otimes [\_]) \circ (S \otimes 1) \circ \Delta (h),
 \]
 where $\mu$ denotes the product in $\mathcal{H}_{par}$. Thus, the map $h \mapsto e_h$ is $K$-linear, in particular $e_{\lambda h}= \lambda e_h$ for all $\lambda \in K$.

Notice that for a general Hopf algebra one also needs to consider the elements $ \tilde{e}_h:=[S(h_{(1)})][h_{(2)}],$ $h\in  \mathcal{H},$ (see \cite{ALVES2015137}), however, since  our  Hopf algebra is cocommutative, we have that
 \begin{align*}
     \tilde{e}_h &= [S(h_{(1)})][h_{(2)}] \\ 
                 &= [S(h_{(1)})][S(S(h_{(2)}))] \\ 
                 &= [S(h)_{(1)}][S(S(h)_{(2)})] \\
                 &= e_{S(h)}.
 \end{align*}

 The following is \cite[Lemma 4.7]{ALVES2015137} stated for our particular case of a cocommutative Hopf algebra.

 \begin{lemma} \label{l_Bprop} 
    For every $h,k \in \mathcal{H}$ the following properties hold:
     \begin{enumerate}[(i)]
         \item $e_k[h]=[h_{(2)}]e_{S(h_{(1)})k}$, in particular $e_{h_{(1)}}[h_{(2)}]=[h]$;
         \item $[h]e_k = e_{h_{(1)}k}[h_{(2)}]$, in particular $[h_{(1)}]e_{S(h_{(2)})}=[h]$;
         \item $e_{h_{(1)}}e_{h_{(2)}} = e_{h}$;
         %\item $\tilde{e}_k [h] =[h_{(1)}]\tilde{e}_{kh_{(2)}}$;
         %\item $[h]\tilde{e}_k=\tilde{e}_{kS(h_{(2)})}[h_{(1)}]$;
         %\item $\tilde{e}_h e_k = e_k \tilde{e}_h$, consequently $e_h e_k = e_k e_h$.
				 \item  $e_h e_k = e_k e_h.$
     \end{enumerate}
 \end{lemma}

 \begin{definition}
     We define $\mathcal{B}$ as the subalgebra of $\mathcal{H}_{par}$ generated by
     $\{ e_h \mid h \in \mathcal{H}\}.$ 
		%= \{\tilde{e}_h \mid h \in \mathcal{H} \}.
      \end{definition}

 The following lemma is a consequence of \cite[Theorem 4.8]{ALVES2015137} and Proposition~\ref{prop:CorollaryFromAlves-Batista-Vercruyse}.
 
 \begin{lemma} \label{l_HparBstructure}
     The algebra $\mathcal{B}$ is a left $\mathcal{H}_{par}$-module with the action
     \begin{equation} \label{e_LeftActionB}
         [h] \triangleright b:= [h_{(1)}]b[S(h_{(2)})].
     \end{equation}
 \end{lemma}
 
 Since $\mathcal{H}$ is cocommutative then we can consider the antipode $S$ as an isomorphism between $\mathcal{H}$ and $\mathcal{H}^{op}.$ Furthermore,  it determines a partial representation
 \begin{align*}
     \mathcal{S}': \mathcal{H} & \to \mathcal{H}_{par}^{op}, \\
                     h &\mapsto [S(h)],
 \end{align*}
 and thus there  exists a morphism of algebras 
 \begin{equation} \label{e_Spar}
     \mathcal{S}: \mathcal{H}_{par} \to \mathcal{H}_{par}^{op},
 \end{equation}
 such that $\mathcal{S}([h])=[S(h)]$. Since, $S: \mathcal{H} \to \mathcal{H}^{op}$ is an algebra isomorphism, then we have an algebra isomorphism $\mathcal{H}_{par} \overset{\mathcal{S}}{\cong} \mathcal{H}_{par}^{op}$. Consequently, we obtain the following lemma.
 
 \begin{lemma} \label{l_HparMODHpar}
     The categories $\mathcal{H}_{par}$-\textbf{Mod} and \textbf{Mod}-$\mathcal{H}_{par}$ are isomorphic. In particular, given a left $\mathcal{H}_{par}$-module $V$, then $V$ is a right $\mathcal{H}_{par}$-module with action 
     \[
         x \triangleleft [h] := [S(h)] \triangleright x.
     \]
     Analogously, if $V$ is a right $\mathcal{H}_{par}$-module, then $V$ is a left $\mathcal{H}_{par}$-module with action 
     \[
         [h] \triangleright x := x \triangleleft [S(h)].
     \]
 \end{lemma}
 \begin{proof}
     Obviously,   $B$-\textbf{Mod} $\cong$ \textbf{Mod}-$B^{op}$ for any algebra $B,$ and since $\mathcal{H}_{par} \overset{\mathcal{S}}{\cong} \mathcal{H}_{par}^{op}$, then \textbf{Mod}-$\mathcal{H}_{par}^{op}\cong$ \textbf{Mod}-$\mathcal{H}_{par}$.
 \end{proof}

 In particular, by Lemmas~\ref{l_HparBstructure} and \ref{l_HparMODHpar} we conclude that $\mathcal{B}$ is a right 
 $\mathcal{H}_{par}$-module with action
 \begin{equation} \label{e_RightActionB}
    b \triangleleft [h] := [S(h_{(1)})]b[h_{(2)}].
 \end{equation} 
 Indeed, 
 \begin{equation*}
     b \triangleleft [h] := [S(h)] \triangleright b 
      =[S(h)_{(1)}]b[S(S(h)_{(2)})] = 
      [S(h_{(2)})]b[S^2(h_{(1)})] =[S(h_{(1)})]b[h_{(2)}],
 \end{equation*}
 thanks to the cocommutativity of $\mathcal{H}.$
 
 \begin{lemma} \label{l_BactionB}
     For any $w, u \in \mathcal{B}$ we have that $w \triangleright u = wu$.
 \end{lemma}
 \begin{proof}
     Let $h \in \mathcal{H}$. Then, keeping in mind the cocommutativity of  $\mathcal{H}$ and Lemma~\ref{l_Bprop},  we have
     \begin{align*}
         e_h \triangleright u
         &= \sum [h_{(1)}][S(h_{(2)})] \triangleright  u \\ 
         &= \sum [h_{(1)}] \triangleright ([S(h_{(2)})]u[h_{(3)}]) \\ 
         &= \sum [h_{(1)}][S(h_{(3)})]u[h_{(4)}][S(h_{(2)})] \\ 
         &= \sum [h_{(1)}][S(h_{(2)})]u[h_{(3)}][S(h_{(4)})] \\ 
         &= \sum e_{h{(1)}}ue_{h{(2)}} \\ 
         &= \sum e_{h{(1)}}e_{h{(2)}}u \\ 
         &= e_hu.
     \end{align*} 
     The statement with an arbitrary $w \in \mathcal{B}$ follows immediately. 
    \end{proof}
 
 \begin{proposition} \label{p_AeMAction}
    Let $M$ be an $A \#\mathcal{H}$-bimodule. Then the map $\pi: \mathcal{H} \to \operatorname{End}_K(A \otimes_{A^e}M)$ such that
     \begin{equation} \label{e_AeMKpGM}
         \pi_h(a \otimes_{A^e}  m):=[h_{(1)}] \triangleright a \otimes_{A^e} [h_{(2)}] \triangleright m,
     \end{equation}
     is a partial representation on $\mathcal{H}$. In particular $A \otimes_{A^e}M$ is a $\mathcal{H}_{par}$-module with the action
     \begin{equation}
         [h] \triangleright (a \otimes_{A^e}  m):=[h_{(1)}] \triangleright a \otimes_{A^e} [h_{(2)}] \triangleright m.
     \end{equation}
 
 \end{proposition}
 \begin{proof}
     First, we have to verify that the map
 \begin{align*}
     \pi_h: A \otimes_{A^e} M  &\to A \otimes_{A^e} M \\ 
         a \otimes_{A^e}m &\mapsto [h_{(1)}] \triangleright a \otimes_{A^e} [h_{(2)}] \triangleright m
 \end{align*}
     is well-defined. Indeed, for all $b,c \in A$ we obtain using Lemma \ref{l_envl} and the cocommutativity of $\mathcal{H}$ that
     \begin{align*}
         [h_{(1)}] \triangleright &( a \cdot (c \otimes b)) \otimes_{A^e} [h_{(2)}] \triangleright m \\
         &= [h_{(1)}] \triangleright (b a c) \otimes_{A^e} [h_{(2)}] \triangleright m \\ 
         &= h_{(1)} \cdot (b a c) \otimes_{A^e} [h_{(2)}] \triangleright m \\ 
         & = (h_{(1)} \cdot b )(h_{(2)} \cdot a) (h_{(3)} \cdot c) \otimes_{A^e} [h_{(4)}] \triangleright m \\ 
         & = (h_{(4)} \cdot b )(h_{(1)} \cdot a) (h_{(2)} \cdot c) \otimes_{A^e} [h_{(3)}] \triangleright m \\ 
         & = (h_{(1)} \cdot a) \otimes_{A^e} (h_{(2)} \cdot c )([h_{(3)}] \triangleright m)(h_{(4)} \cdot b) \\
         & = (h_{(1)} \cdot a) \otimes_{A^e} ((h_{(2)} \cdot c) \# 1_{\mathcal{H}})( 1_A \# h_{(3)}) \cdot  m \cdot  (1_A \# S(h_{(4)}))((h_{(5)} \cdot b) \# 1_{\mathcal{H}}) \\
        (\text{by Lemma \ref{l_commutSmash}}) & = (h_{(1)} \cdot a) \otimes_{A^e} ((h_{(2)} \cdot c) \# h_{(3)}) \cdot m \cdot ( b \# 1_{\mathcal{H}})(1_A \# S(h_{(4)})) \\
        & = (h_{(1)} \cdot a) \otimes_{A^e} (1_A \# h_{(2)})(c \# 1_{\mathcal{H}}) \cdot m \cdot ( b \# 1_{\mathcal{H}})(1_A \# S(h_{(3)})) \\
        & = (h_{(1)} \cdot a) \otimes_{A^e} (1_A \# h_{(2)})(c \cdot m \cdot b )(1_A \# S(h_{(3)})) \\ 
        & = ([h_{(1)}] \triangleright a) \otimes_{A^e} [h_{(2)}] \triangleright ((c\otimes b) \cdot m ).
     \end{align*}
     Now observe that
     \[
         \pi_h (a \otimes_{A^e}m)= \otimes_{A^e} \circ (- \triangleright a \otimes - \triangleright m) \circ [\_] \otimes [\_] \circ \Delta(h),
     \]
     where $(- \triangleright a \otimes - \triangleright m )(x \otimes y):= x\triangleright a \otimes y \triangleright m$, for all $a \in A$, $m \in M$ and $x,y \in \mathcal{H}_{par}$. Therefore, $\pi$ is a $K$-linear map. Furthermore, since we defined $\pi$ using the $\mathcal{H}_{par}$-module (partial representation of $\mathcal{H}$) structures of $A$ and $M$ we have that $\pi$ is a partial representation of $\mathcal{H}$.
 \end{proof}
 
 \begin{remark}
     Let $f:X \to Y$ be a map of $A \# \mathcal{H} $-bimodules. Then, by Proposition~\ref{p_AskewGmIsKparGm}, $f$ is  a map of $\mathcal{H}_{par}$-modules, and, therefore, $1_A \otimes_{A^e} f : A \otimes_{A^e}X \to A \otimes_{A^e} Y$ is a map of $\mathcal{H}_{par}$-modules.
 \end{remark}
 
 \underline{In what follows} $M$ will be an $A \#\mathcal{H}$-bimodule, and we shall only consider the left $\mathcal{H}_{par}$-module structure on $A \otimes_{A^e} M$ defined by (\ref{e_AeMKpGM}).
 
 By Proposition \ref{p_AeMAction} we can define the (covariant) right exact functor 
 \[
   F_1(-):= A \otimes_{A^e}-: (A \# \mathcal{H})^e\!\operatorname{-Mod}  \to \mathcal{H}_{par}\!\operatorname{-Mod},    
 \]
 and by \eqref{e_RightActionB}  the right exact functor
 \[
     F_2(-):= \mathcal{B} \otimes_{\mathcal{H}_{par}}-:= \mathcal{H}_{par}\!\operatorname{-Mod} \to K \!\operatorname{-Mod}. 
 \]
 Recall that the Hochschild homology of $A \# \mathcal{H}$ with coefficients in $M$ is the left derived functor of 
 \[
     F(-):=A\# \mathcal{H} \otimes_{(A\# \mathcal{H})^e}-:(A \# \mathcal{H})^e\!\operatorname{-Mod} \to K\!\operatorname{-Mod}.
 \]
 
 The following is Proposition $1.4$ of \cite{DeMeyerSeparableAlgebras}.
 
 \begin{proposition} \label{p_DMSA}
     Let $R$ and $S$ be rings and $f: R \to S$ a homomorphism of rings such that $S$ is a projective $R$-module. Then, any projective $S$-module is a projective $R$-module.
 \end{proposition}
 
 Using a straightforward dual basis argument one easily obtains the next:
 
 \begin{lemma} \label{lemma:tensorProjectives} 
    Let $R$ and $S$ be unital $K$-algebras. Suppose that $X$ is a projective left $R$-module and $Y$ is a projective right $S$-module. Then $X \otimes _K Y$ is a projective left $R\otimes _K S^{op}$-module.
 \end{lemma}
 
 \begin{lemma} \label{l_AskewP} Suppose that $\mathcal{H}$ is projective over $K.$ Then 
     $(A \# \mathcal{H})^e$ is a projective left $A^e$-module.
 \end{lemma}

 \begin{proof} 
 Since $\mathcal{H}$ is projective over $K,$ it is a direct summand of a free module, and tensoring by $A$ we readily see that  $A \otimes \mathcal{H}$ is a projective left $A$-module. Then, $A \# \mathcal{H} =  (A \otimes \mathcal{H})(1_A \otimes 1_{\mathcal{H}})$ is  projective over $A$ as a left module, being a direct summand of  $A \otimes \mathcal{H}.$   In view of Lemma~\ref{p_DMSA} it remains to show that $A \# \mathcal{H}$ is projective as a right $A$-module.
 
 Let $\{h_i, f_i\}_{i\in I}$ be a dual basis for $\mathcal{H}$ over $K$, where $I$ is some index set. Then 
 $h = \sum _{i \in I} f_i(h) h_i,$ for each $h\in \mathcal{H}.$ Define the mappings $g_i: A\otimes \mathcal{H} \to A,$ $i\in I,$ by 
 $$g_i(a \otimes h)= f_i(h_{(1)}) (S(h_{(2)}) \cdot a).$$ Then for all 
 $a, a' \in A, h\in \mathcal{H}$ and $i\in I$ we have, using cocommutativity of $\mathcal{H}$,  that
 \begin{align*} 
 g_i ((a\otimes h) (a'\otimes 1_H))
 &= g_i(a (h_{(1)} \cdot a') \otimes h_{(2)}) \\ 
 &=f_i(h_{(1)}) S(h_{(2)}) \cdot (a (h_{(3)} \cdot a'))
 f_i(h_{(1)}) (S(h_{(2)}) \cdot a)  (S(h_{(3)}) \cdot (h_{(4)} \cdot a'))\\
  &= f_i(h_{(1)}) \underbrace{(S(h_{(2)}) \cdot a)  (S(h_{(3)}) \cdot 1_A)}
 ((\underbrace{S(h_{(4)})  h_{(5)}}) \cdot a')\\
 &= f_i(h_{(1)}) (S(h_{(2)}) \cdot a) a'= g_i(a \otimes h)a'.
\end{align*}
Thus, each $g_i$ is a map of right $A$-modules. In particular,
$$g_i (a \# h) = g_i ((a \otimes h) (1_A \otimes 1_H))= 
g_i (a \otimes h) 1_A = g_i (a \otimes h),$$ for all 
$a\in A, h \in \mathcal{H}.$ Next we  show that
$\{1_A\otimes h_i, g_i\}_{i \in I}$ is a dual basis for the right $A$-module $A\# {\mathcal H},$ which will complete our proof. Indeed, for each $a\in A, h \in \mathcal{H},$ using that our partial action is symmetric and the cocommutativity of  $\mathcal{H},$  we see that
\begin{align*}
\sum _{i \in I}   (1_A\otimes h_i) g_i(a \# h)
&= \sum _{i \in I}   (1_A\otimes h_i) f_i(h_{(1)}) (S(h_{(2)}) \cdot a)\\
&=\sum _{i \in I}   (1_A\otimes h_i) 
(f_i(h_{(1)}) (S(h_{(2)}) \cdot a) \otimes 1_{\mathcal{H}})\\
&=   (1_A\otimes \sum _{i \in I}f_i(h_{(1)})h_i) 
( (S(h_{(2)}) \cdot a) \otimes 1_{\mathcal{H}})\\
&=   (1_A\otimes h_{(1)}) 
( (S(h_{(2)}) \cdot a) \otimes 1_{\mathcal{H}})=
    h_{(1)}\cdot  (S(h_{(2)}) \cdot a) \otimes h_{(2)} \\
&=( h_{(1)}  S(h_{(2)}) \cdot a) (h_{(3)}\cdot  1_A) \otimes h_{(4)}=  a (h_{(1)}\cdot  1_A) \otimes h_{(2)} = a \#h,
\end{align*} 
as desired. 
\end{proof}

 From Proposition \ref{p_DMSA} and Lemma \ref{l_AskewP} we obtain the following:
 \begin{remark} \label{r_projectiveAskewGmodulesAreAprojectives}
    If $\mathcal{H}$ is projective over $K$, then any projective  $(A \# \mathcal{H})$-bimodule is a projective $A$-bimodule. Therefore, any projective resolution of $M$ in $(A \# \mathcal{H})^e \!\operatorname{-Mod} $ is a projective resolution of $M$ in $A^e\!\operatorname{-Mod}$. Thus, considering $M$ as an $A$-bimodule, the left derived functor of $F_1$ computes the Hochschild homology of $M,$  i.e.,
     \[
         H_\bullet(A,M) \cong L_\bullet F_1(M).
     \]
 \end{remark}
 
 \begin{lemma} \label{l_teneq}
     Let $X$ be an $A \# \mathcal{H}$-bimodule and $Y$ a $\mathcal{H}_{par}$-module, then
     \begin{enumerate}[(i)]
         \item For all $h \in \mathcal{H}$ we have that $h \cdot 1_A$ is central in $A$.
         \item $e_h \triangleright a = (h \cdot 1_A) a$, for all $h \in \mathcal{H}$ and $a \in A$, considering $A$ with the $\mathcal{H}_{par}$-module structure given by (\ref{e_Hpar-modA}). Then, $e_h \triangleright 1_A = h \cdot 1_A$ and $w \triangleright a = (w \triangleright 1_A)a$ for all $a \in A$ and $w \in \mathcal{B}$;
         \item $e_h \triangleright x = (h_{(1)}\cdot 1_A) \cdot x \cdot (h_{(2)}\cdot 1_A)$ $h \in \mathcal{H}$ and $x \in X$, considering $X$ with the $\mathcal{H}_{par}$-module structure given by (\ref{e_Hpar-modM});
         \item $e_h \otimes_{\mathcal{H}_{par}} y = 1_{\mathcal{B}} \otimes_{\mathcal{H}_{par}} [S(h)] \cdot y = 1_{\mathcal{B}} \otimes_{\mathcal{H}_{par}} e_h \cdot y$, as elements of $\mathcal{B} \otimes_{\mathcal{H}_{par}}Y$, for all $h \in \mathcal{H}$ and $y \in Y$,
         \item $e_h \triangleright (a \otimes_{A^e} x) = a \otimes_{A^e} e_h \triangleright x = e_h \triangleright a \otimes_{A^e}  x$, as elements of $A \otimes_{A^e}X$; so that $w \triangleright (a \otimes_{A^e} x) = w \triangleright a \otimes_{A^e}  x=a \otimes_{A^e} w \triangleright x$, for all $w \in \mathcal{B}$.
     \end{enumerate}
 \end{lemma}
 
 \begin{proof} \,
     \begin{enumerate}[(i)]
         \item By direct computations using the cocommutativity and the symmetry of the partial action we obtain
         \begin{align*}
             (h \cdot 1_A)a  &= \varepsilon(h_{(1)})(h_{(2)} \cdot 1_A)a \\ 
                             &= (h_{(2)} \cdot 1_A)((\varepsilon(h_{(1)})1_{\mathcal{H}}) \cdot a) \\
                             &= (h_{(1)} \cdot 1_A)((h_{(2)}S(h_{(3)})) \cdot a) \\
                             &= h_{(1)}\cdot (S(h_{(2)}) \cdot a) \\
                             &= (h_{(1)}S(h_{(3)}) \cdot a)(h_{(2)} \cdot 1_A) \\
                             &= (h_{(1)}S(h_{(2)}) \cdot a)(h_{(3)} \cdot 1_A) \\
                             &= ((\varepsilon(h_{(1)})1_{\mathcal{H}}) \cdot a)(h_{(2)} \cdot 1_A) \\
                             &= a (\varepsilon(h_{(1)})h_{(2)} \cdot 1_A) \\
                             &= a (h \cdot 1_A).
         \end{align*}
         \item For any $h \in \mathcal{H}$ and $a \in A$ we obtain 
         \begin{align*}
             e_h \triangleright a &= [h_{(1)}][S(h_{(2)})] \triangleright a \\ 
                     &= h_{(1)}\cdot (S(h_{(2)}) \cdot a) \\ 
                     &= (h_{(1)}\cdot 1_A) (h_{(2)}S(h_{(3)}) \cdot a) \\ 
                     &= (h_{(1)}\cdot 1_A) (\varepsilon(h_{(2)})1_{\mathcal{H}} \cdot a) \\ 
                     &= (h \cdot 1_A)a.
         \end{align*}
         \item Let $h \in \mathcal{H}$ and $x \in X$. Then,
         \begin{align*}
             e_h \triangleright x    &= (1_A \# h_{(1)})(1_A \# S(h_{(3)})) \cdot x \cdot (1_A \# h_{(4)})(1_A \# S(h_{(2)})) \\
             &= (h_{(1)} \cdot 1_A \# h_{(2)}S(h_{(4)})) \cdot x \cdot (h_{(5)} \cdot 1_A \# h_{(6)}S(h_{(3)})) \\
             &= (h_{(1)} \cdot 1_A \# h_{(2)}S(h_{(3)})) \cdot x \cdot (h_{(4)} \cdot 1_A \# h_{(5)}S(h_{(6)})) \\
             &= (h_{(1)} \cdot 1_A \# \varepsilon(h_{(2)})1_{\mathcal{H}}) \cdot x \cdot (h_{(3)} \cdot 1_A \# \varepsilon(h_{(4)})1_{\mathcal{H}}) \\
             &= (h_{(1)} \cdot 1_A \# 1_{\mathcal{H}}) \cdot x \cdot (h_{(2)} \cdot 1_A \# 1_{\mathcal{H}}) \\
             &= (h_{(1)} \cdot 1_A) \cdot x \cdot (h_{(2)} \cdot 1_A).
         \end{align*}
         \item For $h \in \mathcal{H}$ and $y \in Y$ we have that
         \[
             e_h \otimes_{\mathcal{H}_{par}} y = [h_{(1)}][S(h_{(2)})]  \otimes_{\mathcal{H}_{par}} y = 1_{\mathcal{B}} \triangleleft[S(h)] \otimes_{\mathcal{H}_{par}} y = 1_{\mathcal{B}} \otimes_{\mathcal{H}_{par}} [S(h)] \cdot y 
         \]
         and
         \begin{align*}
             e_h \otimes_{\mathcal{H}_{par}} y &= e_{h_{(1)}}e_{h_{(2)}} \otimes_{\mathcal{H}_{par}} y \\ 
             &= [h_{(1)}][S(h_{(2)})]1_{\mathcal{B}}[h_{(3)}][S(h_{(4)})] \otimes_{\mathcal{H}_{par}} y \\
             &= [h_{(1)}](1_{\mathcal{B}} \triangleleft [h_{(2)}])[S(h_{(3)})] \otimes_{\mathcal{H}_{par}} y \\ 
             &= [h_{(1)}](1_{\mathcal{B}} \triangleleft [h_{(3)}])[S(h_{(2)})] \otimes_{\mathcal{H}_{par}} y \\ 
             &= (1_{\mathcal{B}} \triangleleft [h_{(2)}]) \triangleleft [S(h_{1})] \otimes_{\mathcal{H}_{par}} y \\ 
             &= (1_{\mathcal{B}} \triangleleft [h_{(1)}][S(h_{2})]) \otimes_{\mathcal{H}_{par}} y \\ 
             &= (1_{\mathcal{B}} \triangleleft e_h) \otimes_{\mathcal{H}_{par}} y \\
             &= 1_{\mathcal{B}} \otimes_{\mathcal{H}_{par}} e_h \cdot y.
         \end{align*}
         \item For $a \in A$ and $x \in X$ we obtain
         \begin{align*}
             e_h \triangleright (a \otimes_{A^e} x) &= [h_{(1)}] \triangleright ([S(h_{(2)})] \triangleright a \otimes_{A^e}[S(h_{(3)})] \triangleright x) \\
             &= [h_{(1)}][S(h_{(3)})] \triangleright a \otimes_{A^e} [h_{(2)}][S(h_{(4)})] \triangleright x \\ 
             &= [h_{(1)}][S(h_{(2)})] \triangleright a \otimes_{A^e} [h_{(3)}][S(h_{(4)})] \triangleright x \\
             &= e_{h_{(1)}} \triangleright a \otimes_{A^e} e_{h_{(2)}} \triangleright x \\
            (\text{by } (ii) \text{ and } (iii)) &= (h_{(1)}\cdot 1_A)a \otimes_{A^e} (h_{(2)}\cdot 1_A) \cdot x \cdot (h_{(3)}\cdot 1_A)
         \end{align*}
         Now we have that
         \begin{align*}
             (h_{(1)}\cdot 1_A)a \otimes_{A^e} (h_{(2)}\cdot 1_A) \cdot x \cdot (h_{(3)}\cdot 1_A)
             &= a \otimes_{A^e} (h_{(1)}\cdot 1_A) \cdot x \cdot (h_{(2)}\cdot 1_A)(h_{(3)}\cdot 1_A) \\
             &= a \otimes_{A^e} (h_{(1)}\cdot 1_A) \cdot x \cdot (h_{(2)}\cdot 1_A) \\
             &= a \otimes_{A^e} e_h \triangleright x,
         \end{align*}
         and on the other hand
         \begin{align*}
             (h_{(1)}\cdot 1_A)a \otimes_{A^e} (h_{(2)}\cdot 1_A) \cdot x \cdot (h_{(3)}\cdot 1_A)
             &= (h_{(3)}\cdot 1_A)(h_{(1)}\cdot 1_A)a(h_{(2)}\cdot 1_A) \otimes_{A^e}  x  \\
             (\text{by }(i))&= (h_{(1)}\cdot 1_A)(h_{(2)}\cdot 1_A)(h_{(3)}\cdot 1_A)a \otimes_{A^e}  x  \\
             &= (h\cdot 1_A)a \otimes_{A^e}  x  \\
            (\text{by }(ii)) &= e_h \triangleright a \otimes_{A^e}  x.
         \end{align*}
     \end{enumerate}
 \end{proof}

 \begin{proposition}
     The algebra $A \# \mathcal{H}$ is a $\mathcal{H}_{par}$-module with action, given by 
    \begin{equation} \label{e_KparG-mod-ArxG}
     [h] \triangleright (a\# k) := (1_A \# h) (a\# k).
    \end{equation}
    In particular
    \begin{equation} \label{e_ehbysmash}
     e_h \triangleright (a\# k) = ((h \cdot 1_A) a\# k) = (e_h \triangleright a)\# k.
    \end{equation}
 \end{proposition}
 
 \begin{proof}
     We know that $\pi_0:\mathcal{H} \to A \# \mathcal{H}$ such that $\pi_0(h):=1_A \# h$ is a partial representation of $\mathcal{H}$, whence we conclude that $A \# \mathcal{H}$ is a $\mathcal{H}_{par}$-module with action (\ref{e_KparG-mod-ArxG}). Finally, by direct computation we obtain
 
     \begin{align*}
         e_h \triangleright (a\# k) &= (1_A \# h_{(1)})(1_A \# S(h_{(2)}))(a\# k) \\ 
         &= (h_{(1)}\cdot 1_A \# h_{(2)}S(h_{(3)}))(a\# k) \\ 
         &=(h_{(1)}\cdot 1_A \# \varepsilon(h_{(2)})1_{\mathcal{H}})(a\# k) \\
         &=(h \cdot 1_A \# 1_{\mathcal{H}})(a\# k) \\
         &=(h \cdot 1_A) a\# k \\
         &=(e_h \triangleright a)\# k,
     \end{align*}
     where the final equality is due $(ii)$ of Lemma \ref{l_teneq}.
 \end{proof}
 
 \begin{proposition} \label{p_ArtG-Mod-XtensorArtG}
     Let $X$ be a right $\mathcal{H}_{par}$-module. Then, $X \otimes_{\mathcal{B}} (A\# \mathcal{H})$ is an $A\# \mathcal{H}$-bimodule with actions:
     \[
          a \# h \cdot (x \otimes_{\mathcal{B}} c \# t) := x \cdot [S(h_{(1)})] \otimes_{\mathcal{B}} (a \# h_{(2)})(c \# t)
     \]
     and
     \[
     (x \otimes_{\mathcal{B}} c \#t) \cdot a \# h:= x \otimes_{\mathcal{B}} (c \#t)(a \# h).
     \]
     Thus, we obtain a functor $- \otimes_{\mathcal{B}} A \# \mathcal{H}: \operatorname{Mod-}\!\mathcal{H}_{par} \to (A\# \mathcal{H})^e \!\operatorname{-Mod}$.
 \end{proposition}
 
 \begin{proof}
     The right action is well-defined since it is just the action induced by the right multiplication on $A\# \mathcal{H}$. For the left action we want to verify that $a \# h \cdot (x \cdot e_k\otimes_{\mathcal{B}} c \# t) = a \# h \cdot (x \otimes_{\mathcal{B}} e_k \triangleright(c \# t))$, i.e.,
     \[
         (x \cdot e_k)\cdot[S(h_{(1)})] \otimes_{\mathcal{B}}(a \# h_{(2)})(c \# t) = x\cdot[S(h_{(1)})] \otimes_{\mathcal{B}}(a \# h_{(2)})(e_k \triangleright (c \# t)).
     \]
     Starting the computations with the right part of the above equation we have
     \begin{align*}
         x\cdot[S(h_{(1)})] &\otimes_{\mathcal{B}}(a \# h_{(2)})(e_k \triangleright (c \# t))  \\
         &= x\cdot[S(h_{(1)})] \otimes_{\mathcal{B}}(a \# h_{(2)})(1_A \# k_{(1)})(1_A \# S(k_{(2)}))(c \# t) \\
         &= x\cdot[S(h_{(1)})] \otimes_{\mathcal{B}}(a \# h_{(2)})(k_{(1)} \cdot 1_A \# k_{(2)}S(k_{(3)}))(c \# t) \\
         &= x\cdot[S(h_{(1)})] \otimes_{\mathcal{B}}(a \# h_{(2)})(k_{(1)} \cdot 1_A \# \varepsilon(k_{(2)})1_{\mathcal{H}})(c \# t) \\
         &= x\cdot[S(h_{(1)})] \otimes_{\mathcal{B}}(a \# h_{(2)})(\varepsilon(k_{(2)})k_{(1)} \cdot 1_A \# 1_{\mathcal{H}})(c \# t) \\
         &= x\cdot[S(h_{(1)})] \otimes_{\mathcal{B}}(a \# h_{(2)})(k \cdot 1_A \# 1_{\mathcal{H}})(c \# t) \\
         &= x\cdot[S(h_{(1)})] \otimes_{\mathcal{B}}(a (h_{(2)} \cdot (k \cdot 1_A)) \# h_{(3)})(c \# t) \\
         &= x\cdot[S(h_{(1)})] \otimes_{\mathcal{B}}\big(a (h_{(2)} \cdot 1_A)(h_{(3)}k \cdot 1_A) \# h_{(4)}\big)(c \# t) \\
         (\text{by (\ref{e_ehbysmash})})&= x\cdot[S(h_{(1)})] \otimes_{\mathcal{B}} e_{h_{(2)}} \triangleright \big((a (h_{(3)}k \cdot 1_A) \# h_{(4)})(c \# t)\big) \\
         &= x\cdot[S(h_{(1)})]e_{h_{(2)}} \otimes_{\mathcal{B}} \big((a (h_{(3)}k \cdot 1_A) \# h_{(4)})(c \# t)\big) \\
         (\text{by Lemma \ref{l_Bprop} }(ii)) &= x\cdot[S(h_{(1)})] \otimes_{\mathcal{B}} \big((a (h_{(2)}k \cdot 1_A) \# h_{(3)})(c \# t)\big) \\
         (\text{by (\ref{e_ehbysmash})}) &= x\cdot[S(h_{(1)})] \otimes_{\mathcal{B}} e_{h_{(2)}k} \triangleright \big((a \# h_{(3)})(c \# t)\big) \\ 
         &= x\cdot[S(h_{(1)})]e_{h_{(2)}k} \otimes_{\mathcal{B}} \big((a \# h_{(3)})(c \# t)\big) \\
         (\text{by Lemma \ref{l_Bprop} }(i)) &= x\cdot e_k [S(h_{(1)})] \otimes_{\mathcal{B}} \big((a \# h_{(2)})(c \# t)\big) \\ 
         &= (x \cdot e_k)\cdot[S(h_{(1)})] \otimes_{\mathcal{B}}(a \# h_{(2)})(c \# t).
     \end{align*}
     Then, $a \# h \cdot (x \otimes_{\mathcal{B}} c \# t) = x \cdot [S(h_{(1)})] \otimes_{\mathcal{B}} (a \# h_{(2)})(c \# t)$ is well-defined. Now observe that
     \begin{align*}
         (b\# k) \cdot \bigl( a \# h \cdot (x \otimes_{\mathcal{B}} c \# t) \bigr) 
         &= (b\# k) \cdot \bigl(  x \cdot [S(h_{(1)})] \otimes_{\mathcal{B}} (a \# h_{(2)})(c \# t) \bigr) \\ 
         &= x \cdot [S(h_{(1)})][S(k_{(1)})] \otimes_{\mathcal{B}} (b \# k_{(2)})(a \# h_{(2)})(c \# t) \\
        (\text{By Lemma \ref{l_Bprop} }(ii)) &= x \cdot [S(h_{(1)})][S(k_{(1)})]e_{k_{(2)}} \otimes_{\mathcal{B}} (b \# k_{(3)})(a \# h_{(2)})(c \# t) \\
         &= x \cdot [S(h_{(1)})S(k_{(1)})]e_{k_{(2)}} \otimes_{\mathcal{B}} (b \# k_{(3)})(a \# h_{(2)})(c \# t) \\
         &= x \cdot [S(h_{(1)})S(k_{(1)})] \otimes_{\mathcal{B}} e_{k_{(2)}} \triangleright \bigl((b \# k_{(3)})(a \# h_{(2)})(c \# t)\bigr) \\
         &= x \cdot [S(k_{(1)}h_{(1)})] \otimes_{\mathcal{B}} e_{k_{(2)}} \triangleright \bigl((b \# k_{(3)})(a \# h_{(2)})(c \# t)\bigr) \\
         &= x \cdot [S(k_{(1)}h_{(1)})] \otimes_{\mathcal{B}} ((k_{(2)}\cdot 1_A)b \# k_{(3)})(a \# h_{(2)})(c \# t) \\
         (\text{by Lemma \ref{lemma from Alvares-Batista} } (ii))&= x \cdot [S(k_{(1)}h_{(1)})] \otimes_{\mathcal{B}} (b \# k_{(2)})(a \# h_{(2)})(c \# t) \\
         &= x \cdot [S(k_{(2)}h_{(1)})] \otimes_{\mathcal{B}} (b \# k_{(1)})(a \# h_{(2)})(c \# t) \\
         &= x \cdot [S(k_{(2)}h_{(1)})] \otimes_{\mathcal{B}} (b (k_{(1)}\cdot a) \# k_{(3)}h_{(2)})(c \# t) \\
         &= (b (k_{(1)} \cdot a) \# k_{(2)}h) \cdot (x \otimes_{\mathcal{B}} c \# t) \\ 
         &= (b\# k)(a \# h) \cdot (x \otimes_{\mathcal{B}} c \# t).
     \end{align*} 
     Therefore, the left action is well-defined, and since 
     \[
         \bigl(a \# h \cdot (x \otimes_{\mathcal{B}} c \# t) \bigr)\cdot b \#k = x \cdot [S(h_{(1)})] \otimes_{\mathcal{B}} (a \# h_{(2)})(c \# t)(b \#k) = a \# h \cdot \bigl((x \otimes_{\mathcal{B}} c \# t) \cdot b \#k\bigr).
     \]
     We have that $X \otimes_{\mathcal{B}} A \# \mathcal{H}$ is a $A \# \mathcal{H}$-bimodule. Finally, if $f: X \to Y$ is a map of right $\mathcal{H}_{par}$-modules then $f \otimes_{\mathcal{B}}1: X \otimes_{\mathcal{B}} A \# \mathcal{H} \to Y \otimes_{\mathcal{B}} A \# \mathcal{H}$ is a map of $A \# \mathcal{H}$-bimodules. Indeed, it is clear that $f \otimes_{\mathcal{B}}1$ is a map of right $A \# \mathcal{H}$-modules. On the other hand, observe that for any $x \in X$ and $z \in A \# \mathcal{H}$ we have
     \begin{align*}
         (a \# h) \cdot \big((f \otimes_{\mathcal{B}}1) (x \otimes_{\mathcal{B}} z) \big) & = (a \# h) \cdot (f(x)\otimes_{\mathcal{B}}z) \\ 
         &= f(x) \cdot [S(h_{(1)})] \otimes_{\mathcal{B}} (a \#h_{(2)})z \\
         &= f(x \cdot [S(h_{(1)})]) \otimes_{\mathcal{B}} (a \#h_{(2)})z \\
         &= (f\otimes_{\mathcal{B}}1) \bigl(   x \cdot [S(h_{(1)})] \otimes_{\mathcal{B}} (a \#h_{(2)})z \bigr) \\
         &= (f\otimes_{\mathcal{B}}1) \bigl(  (a \#h) \cdot ( x  \otimes_{\mathcal{B}} z) \bigr).
     \end{align*}
 \end{proof}
 
 \begin{proposition} \label{p_inXB}
     The functors
     \[
         -\otimes_{\mathcal{H}_{par}} (A \otimes_{A^e}-): \operatorname{Mod-}\!\mathcal{H}_{par} \times (A \# \mathcal{H})^e\!\operatorname{-Mod}  \to K\operatorname{-Mod} 
     \]
     and
     \[
         (- \otimes_{\mathcal{B}} A \# \mathcal{H}) \otimes_{(A \# \mathcal{H})^e}-:\operatorname{Mod-}\!\mathcal{H}_{par} \times (A \# \mathcal{H})^e\!\operatorname{-Mod}  \to K\operatorname{-Mod}
     \] 
     are naturally isomorphic.
 \end{proposition}
 
 \begin{proof}
     Let $X$ be a right $\mathcal{H}_{par}$-module and $M$ and $A \# \mathcal{H}$-bimodule. For a fixed $x \in X$ define
     \begin{align*}
         \tilde{\gamma}_{x,M}: A \times M & \to (X \otimes_{\mathcal{B}} A \# \mathcal{H})  \otimes_{(A \# \mathcal{H})^e}M \\ 
         (a,m) &\mapsto (x \otimes_{\mathcal{B}} 1_A \# 1_{\mathcal{H}}) \otimes_{(A \# \mathcal{H})^e} a \cdot m,
     \end{align*} 
     Observe that $\tilde{\gamma}_{x,M}$ is $A^e$-balanced. Indeed, for $d \otimes c \in A^e$ we have
     \begin{align*}
         \tilde{\gamma}_{x,M}(a\cdot(d \otimes c),m) &= \tilde{\gamma}_{x,M}(cad,m) \\ 
                                                 &= (x \otimes_{\mathcal{B}} 1_A \# 1_{\mathcal{H}}) \otimes_{(A \# \mathcal{H})^e} cad \cdot m \\
             (\text{by Lemma } \ref{l_envl})     &= (x \otimes_{\mathcal{B}} 1_A \# 1_{\mathcal{H}})\cdot (c \# 1_{\mathcal{H}}) \otimes_{(A \# \mathcal{H})^e} a\cdot (d \cdot m) \\            
                                                 &= (x \otimes_{\mathcal{B}} c \# 1_{\mathcal{H}}) \otimes_{(A \# \mathcal{H})^e} a\cdot (d \cdot m) \\
                                                 &= (c \# 1_{\mathcal{H}})\cdot (x \otimes_{\mathcal{B}} 1_A \# 1_{\mathcal{H}}) \otimes_{(A \# \mathcal{H})^e} a\cdot (d \cdot m) \\
             (\text{by Lemma } \ref{l_envl})     &= (x \otimes_{\mathcal{B}} 1_A \# 1_{\mathcal{H}}) \otimes_{(A \# \mathcal{H})^e} a\cdot (d \cdot m) \cdot (c \# 1_{\mathcal{H}}) \\
                                                 &= (x \otimes_{\mathcal{B}} 1_A \# 1_{\mathcal{H}}) \otimes_{(A \# \mathcal{H})^e} a\cdot (d \cdot m \cdot c) \\
                                                 &= \tilde{\gamma}_{x,M}(a,(d \otimes c) \cdot m),
     \end{align*}
     so, $\tilde{\gamma}_{x,M}$ is $A^e$-balanced. Therefore, the following map is well-defined
     \begin{align*}
         \gamma_{x,M}: A \otimes_{A^e} M & \to (X \otimes_{\mathcal{B}} A \# \mathcal{H})  \otimes_{(A \# \mathcal{H})^e}M \\ 
         a \otimes_{A^e} m &\mapsto (x \otimes_{\mathcal{B}} 1_A \# 1_{\mathcal{H}}) \otimes_{(A \# \mathcal{H})^e} a \cdot m.
     \end{align*} 
     Now define the function
     \begin{align*}
         \tilde{\gamma}_{(X,M)} : X \times (A\otimes_{A^e}M) & \to (X \otimes_{\mathcal{B}} A \#\mathcal{H})  \otimes_{(A \# \mathcal{H})^e}M \\ 
         (x,a\otimes_{A^e}m) &\mapsto \gamma_{x,M}(a\otimes_{A^e}m)=(x \otimes_{\mathcal{B}} 1_A \# 1_{\mathcal{H}}) \otimes_{(A \# \mathcal{H})^e} a \cdot m.
     \end{align*}
     We want to show that  $\tilde{\gamma}_{(X,M)}$ is $\mathcal{H}_{par}$-balanced. Recall the $A \# \mathcal{H}$-bimodule structure of $X \otimes_{\mathcal{B}} A \# \mathcal{H}$ given by Proposition \ref{p_ArtG-Mod-XtensorArtG}. Let $h \in \mathcal{H}$, thus
     \begin{align*}
         \tilde{\gamma}_{(X,M)}(x \cdot [h], &a\otimes_{A^e}m)  \\
         &= (x \cdot [h] \otimes_{\mathcal{B}} 1_A \# 1_{\mathcal{H}}) \otimes_{(A \# \mathcal{H})^e} a \cdot m \\
         &= (x \cdot [h_{(1)}]e_{S(h_{(2)})} \otimes_{\mathcal{B}} 1_A \# 1_{\mathcal{H}}) \otimes_{(A \# \mathcal{H})^e} a \cdot m \\ 
         &= (x \cdot [h_{(1)}] \otimes_{\mathcal{B}} e_{S(h_{(2)})} \triangleright (1_A \# 1_{\mathcal{H}})) \otimes_{(A \# \mathcal{H})^e} a \cdot m \\ 
         (\text{by } \eqref{e_ehbysmash})&= \big(x \cdot [h_{(1)}] \otimes_{\mathcal{B}} (S(h_{(2)}) \cdot 1_A) \# 1_{\mathcal{H}} \big) \otimes_{(A \# \mathcal{H})^e} a \cdot m \\ 
         &= \big(x \cdot [h_{(1)}] \otimes_{\mathcal{B}}(1_A \# S(h_{(2)}))(1_A \# h_{(3)}) \big)\otimes_{(A \# \mathcal{H})^e} a \cdot m \\ 
         &= 1_A \# S(h_{(1)}) \cdot \big(x \otimes_{\mathcal{B}}1_A \# 1_{\mathcal{H}} \big) \cdot (1_A \# h_{(2)})\otimes_{(A \# \mathcal{H})^e} a \cdot m \\ 
         (\text{by Lemma }\ref{l_envl})  &= \big(x \otimes_{\mathcal{B}}1_A \# 1_{\mathcal{H}} \big) \otimes_{(A \# \mathcal{H})^e} (1_A \# h_{(2)}) \cdot (a \cdot m) \cdot (1_A \# S(h_{(1)})) \\ 
         &= \big(x \otimes_{\mathcal{B}}1_A \# 1_{\mathcal{H}} \big) \otimes_{(A \# \mathcal{H})^e} (1_A \# h_{(2)})(a \# 1_{\mathcal{H}}) \cdot m \cdot (1_A \# S(h_{(1)})) \\ 
         &= \big(x \otimes_{\mathcal{B}}1_A \# 1_{\mathcal{H}} \big) \otimes_{(A \# \mathcal{H})^e} (h_{(2)} \cdot a \# h_{(3)}) \cdot m \cdot (1_A \# S(h_{(1)})) \\
         &= \big(x \otimes_{\mathcal{B}}1_A \# 1_{\mathcal{H}} \big) \otimes_{(A \# \mathcal{H})^e} (h_{(2)} \cdot a\# 1_{\mathcal{H}}) (1_A\# h_{(3)}) \cdot m \cdot (1_A \# S(h_{(1)})) \\
         &= \big(x \otimes_{\mathcal{B}}1_A \# 1_{\mathcal{H}} \big) \otimes_{(A \# \mathcal{H})^e} (h_{(1)} \cdot a\# 1_{\mathcal{H}}) (1_A\# h_{(2)}) \cdot m \cdot 1_A \# S(h_{(3)}) \\ 
         &= \big(x \otimes_{\mathcal{B}}1_A \# 1_{\mathcal{H}} \big) \otimes_{(A \# \mathcal{H})^e} ([h_{(1)}] \triangleright a\# 1_{\mathcal{H}}) \cdot ([h_{(2)}] \triangleright m)  \\ 
         &= \tilde{\gamma}_{(X,M)}(x , ([h_{(1)}] \triangleright a)\otimes_{A^e}([h_{(2)}] \cdot m)) \\
         &= \tilde{\gamma}_{(X,M)}(x ,[h] \triangleright (a\otimes_{A^e}m)).
     \end{align*}
     Then, the following map is well-defined
     \begin{align*}
         \gamma_{(X,M)} : X \otimes_{\mathcal{H}_{par}} (A\otimes_{A^e}M) & \to (X \otimes_{\mathcal{B}} A \#\mathcal{H})  \otimes_{(A \# \mathcal{H})^e}M \\ 
         x \otimes_{\mathcal{H}_{par}} (a\otimes_{A^e}m) &\mapsto (x \otimes_{\mathcal{B}} 1_A \# 1_{\mathcal{H}}) \otimes_{(A \# \mathcal{H})^e} a \cdot m.
     \end{align*}
     In order to obtain its inverse map we fix an $m \in M$ and define the map
     \begin{align*}
         \tilde{\psi}_{X,m}: X \times A \# \mathcal{H} & \to  X \otimes_{\mathcal{H}_{par}} (A\otimes_{A^e}M) \\ 
         (x, a \# h)&\mapsto x \otimes_{\mathcal{H}_{par}}(1_A \otimes_{A^e} (a \# h) \cdot m).
     \end{align*}
     Then, $\tilde{\psi}_{X,m}$ is $\mathcal{B}$-balanced. Indeed,
     \begin{align*}
         \tilde{\psi}_{X,m}(x \cdot e_k, a \# h)    &=  x \cdot e_k \otimes_{\mathcal{H}_{par}}(1_A \otimes_{A^e} (a \# h) \cdot m) \\
         &=  x \otimes_{\mathcal{H}_{par}} e_k  \triangleright (1_A \otimes_{A^e} (a \# h) \cdot m) \\
         (\text{by Lemma \ref{l_teneq} } (v)\text{ and }(ii))&=  x \otimes_{\mathcal{H}_{par}} \big( k\cdot 1_A \otimes_{A^e}  (a \# h) \cdot m \big)  \\
         &=  x \otimes_{\mathcal{H}_{par}} \big( 1_A\otimes_{A^e}  (k\cdot 1_A  \# 1_{\mathcal{H}})\cdot((a \# h) \cdot m) \big)  \\
         &=  x \otimes_{\mathcal{H}_{par}} \big( 1_A \otimes_{A^e}  ((k\cdot 1_A )a \# h) \cdot m \big)  \\
         &=\tilde{\psi}_{X,m}(x, (k\cdot 1_A )a \# h) \\
        (\text{by }\eqref{e_ehbysmash}) &=\tilde{\psi}_{X,m}(x, e_k \triangleright (a \# h)).
     \end{align*}
     Then, the following  map is well-defined
     \begin{align*}
         \psi_{X,m}: X \otimes_{\mathcal{B}} A \# \mathcal{H} & \to  X \otimes_{\mathcal{H}_{par}} (A\otimes_{A^e}M)\\ 
         x \otimes_{\mathcal{B}} (a\# h) &\mapsto x \otimes_{\mathcal{H}_{par}}(1_A \otimes_{A^e} (a \# h)\cdot m).
     \end{align*}
     So, we can define the map
     \begin{align*}
         \tilde{\psi}_{(X,M)}: (X \otimes_{\mathcal{B}} A \# \mathcal{H})  \times M  & \to  X \otimes_{\mathcal{H}_{par}} (A\otimes_{A^e}M)\\ 
         (x \otimes_{\mathcal{B}} a \# h, m) &\mapsto \psi_{X,m}(x \otimes_{\mathcal{B}} a \# h)=x \otimes_{\mathcal{H}_{par}}(1_A \otimes_{A^e} (a\# h) \cdot m),
     \end{align*}
     which is $(A \# \mathcal{H})^e$-balanced. Indeed, for the left action, we have
     \begin{align*}
         &\tilde{\psi}_{(X,M)}((b \# k) \cdot (x \otimes_\mathcal{B} (a \# h)), m)  \\ 
         &= \tilde{\psi}_{(X,M)}(x \cdot [S(k_{(1)})] \otimes_\mathcal{B} (b \# k_{(2)})(a \# h), m) \\
         &= x \cdot [S(k_{(1)})] \otimes_{\mathcal{H}_{par}} \bigl( 1_A \otimes_{A^e} (b \# k_{(2)})(a \# h) \cdot m \bigr) \\ 
         &= x \cdot [S(k_{(1)})] \otimes_{\mathcal{H}_{par}} \bigl( b \otimes_{A^e} (1_A \# k_{(2)})(a \# h) \cdot m \bigr) \\ 
         &= x \cdot [S(k_{(1)})] \otimes_{\mathcal{H}_{par}} \bigl( 1_A \otimes_{A^e} (1_A \# k_{(2)})(a \# h) \cdot m \cdot (b \# 1_{\mathcal{H}})\bigr) \\ 
         &= x  \otimes_{\mathcal{H}_{par}} [S(k_{(1)})] \triangleright \bigl( 1_A \otimes_{A^e} (1_A \# k_{(2)})(a \# h) \cdot m \cdot (b \# 1_{\mathcal{H}})\bigr) \\ 
         &= x  \otimes_{\mathcal{H}_{par}} \bigl( [S(k_{(1)})] \triangleright  1_A \otimes_{A^e} [S(k_{(2)})] \triangleright ((1_A \# k_{(3)})(a \# h) \cdot m \cdot (b \# 1_{\mathcal{H}}))\bigr) \\ 
         &= x  \otimes_{\mathcal{H}_{par}} \Bigl( S(k_{(1)}) \cdot  1_A \otimes_{A^e} (1_A \# S(k_{(2)}))(1_A \# k_{(3)})(a \# h) \cdot m \cdot (b \# 1_{\mathcal{H}})(1_A \# k_{(4)})\Bigr) \\ 
         &= x  \otimes_{\mathcal{H}_{par}} \Bigl( S(k_{(1)}) \cdot  1_A \otimes_{A^e} \big((S(k_{(2)}) \cdot 1_A \# (S(k_{(3)})k_{(4)}))(a \# h) \cdot m \cdot (b \# 1_{\mathcal{H}})(1_A \# k_{(5)}) \big) \Bigr) \\ 
         &= x  \otimes_{\mathcal{H}_{par}} \Bigl( S(k_{(1)}) \cdot  1_A \otimes_{A^e} \big((S(k_{(2)}) \cdot 1_A \# \varepsilon(k_{(3)})1_{\mathcal{H}})(a \# h) \cdot m \cdot (b \# 1_{\mathcal{H}})(1_A \# k_{(4)}) \big)\Bigr) \\ 
         &= x  \otimes_{\mathcal{H}_{par}} \Bigl( S(k_{(1)}) \cdot  1_A \otimes_{A^e} \big((S(k_{(2)}) \cdot 1_A \# 1_{\mathcal{H}}) (a \# h) \cdot m \cdot (b \# 1_{\mathcal{H}})(1_A \# k_{(3)})\big)\Bigr) \\ 
         &= x  \otimes_{\mathcal{H}_{par}} \Bigl( (S(k_{(1)}) \cdot  1_A)(S(k_{(2)}) \cdot 1_A) \otimes_{A^e} \big( (a \# h) \cdot m \cdot (b \# 1_{\mathcal{H}})(1_A \# k_{(3)})\big)\Bigr) \\ 
         &= x  \otimes_{\mathcal{H}_{par}} \Bigl( (S(k_{(1)}) \cdot  1_A) \otimes_{A^e} \big( (a \# h) \cdot m \cdot (b \# 1_{\mathcal{H}})(1_A \# k_{(2)})\big)\Bigr) \\ 
         &= x  \otimes_{\mathcal{H}_{par}} \Bigl( 1_A \otimes_{A^e} \big( (a \# h) \cdot m \cdot (b \# 1_{\mathcal{H}})(1_A \# k_{(2)})((S(k_{(1)}) \cdot  1_A) \# 1_{\mathcal{H}})\big)\Bigr) \\ 
         &= x  \otimes_{\mathcal{H}_{par}} \Bigl( 1_A \otimes_{A^e} \big( (a \# h) \cdot m \cdot (b \# 1_{\mathcal{H}})(1_A \# k_{(1)})((S(k_{(2)}) \cdot  1_A) \# 1_{\mathcal{H}})\big)\Bigr) \\ 
         &= x  \otimes_{\mathcal{H}_{par}} \Bigl( 1_A \otimes_{A^e} \big( (a \# h) \cdot m \cdot (b \# 1_{\mathcal{H}})(k_{(1)}\cdot (S(k_{(2)}) \cdot 1_A) \# k_{(3)})\big)\Bigr) \\
         &= x  \otimes_{\mathcal{H}_{par}} \Bigl( 1_A \otimes_{A^e} \big( (a \# h) \cdot m \cdot (b \# 1_{\mathcal{H}})(k_{(1)} \cdot 1_A \# k_{(2)})\big)\Bigr) \\ 
         &= x  \otimes_{\mathcal{H}_{par}} \Bigl( 1_A \otimes_{A^e} \big( (a \# h) \cdot m \cdot (b \# k)\big)\Bigr) \\
         &= \tilde{\psi}_{(X,M)}((x \otimes_\mathcal{B} (a \# h)), m \cdot (b \# k) ),
     \end{align*}
    and, for the right action
    \begin{align*}
        \tilde{\psi}_{(X,M)}((x \otimes_\mathcal{B} (a \# h))\cdot (b \# k), m) 
        &=  \tilde{\psi}_{(X,M)}(x \otimes_\mathcal{B} (a \# h)(b \# k), m) \\
        &= x \otimes_{\mathcal{H}_{par}}(1_A \otimes_{A^e} (a \# h)(b \# k) \cdot m) \\ 
        &= x \otimes_{\mathcal{H}_{par}}(1_A \otimes_{A^e} (a \# h)\cdot ((b \# k) \cdot m)) \\
        &=\tilde{\psi}_{(X,M)}((x \otimes_{\mathcal{B}} a \# h),  (b \# k )\cdot m).
    \end{align*}
     Thus, the map
     \begin{align*}
         \psi_{(X,M)} :  (X \otimes_{\mathcal{B}} A \# \mathcal{H})  \otimes_{(A \# \mathcal{H})^e} M  & \to  X \otimes_{\mathcal{H}_{par}} (A\otimes_{A^e}M) \\ 
         (x \otimes_{\mathcal{B}} a \# h) \otimes_{(A \# \mathcal{H})^e}  m &\mapsto x \otimes_{\mathcal{H}_{par}}(1_A \otimes_{A^e} (a\# h) \cdot m)
     \end{align*}
     is well-defined. Observe that $\gamma_{(X,M)}$ and $\psi_{(X,M)}$ are mutual inverses since
     \begin{align*}
         \psi_{(X,M)}\big( \gamma_{(X,M)}(x \otimes_{\mathcal{H}_{par}}(a\otimes_{A^e}m))\big) &=  \psi_{(X,M)}\big( (x \otimes_{\mathcal{B}} 1_A \# 1_{\mathcal{H}}) \otimes_{(A \# \mathcal{H})^e} a \cdot m\big) \\ 
         &=x \otimes_{\mathcal{H}_{par}}(1_A \otimes_{A^e} a \cdot m) \\
         &=x \otimes_{\mathcal{H}_{par}}(a \otimes_{A^e} m)
     \end{align*}
     and
     \begin{align*}
         \gamma_{(X,M)}\big( \psi_{(X,M)}((x \otimes_{\mathcal{B}} a \# h) \otimes_{(A \# \mathcal{H})^e} m) \big)&= \gamma_{(X,M)} \big(x \otimes_{\mathcal{H}_{par}}(1_A \otimes_{A^e} a \# h \cdot m) \big) \\
         &=  (x \otimes_{\mathcal{B}} 1_A \# 1_{\mathcal{H}}) \otimes_{(A \# \mathcal{H})^e} a \# h  \cdot m \\
         &=(x \otimes_{\mathcal{B}} a \# h) \otimes_{(A \# \mathcal{H})^e} m.
     \end{align*}

     Let $X'$ be another right $\mathcal{H}_{par}$-module and $M'$ be a $A \# \mathcal{H}$-bimodule, $f: X \to X'$ a map of right $\mathcal{H}_{par}$-modules and $g: M \to M'$ a map of $A \# \mathcal{H}$-bimodules. Then, the following diagram commutes
     \[\begin{tikzcd}
         {X \otimes_{\mathcal{H}_{par}} (A\otimes_{A^e}M)} & {X' \otimes_{\mathcal{H}_{par}} (A\otimes_{A^e}M')} \\
         {(X \otimes_{\mathcal{B}} A \# \mathcal{H})  \otimes_{(A \# \mathcal{H})^e}M} & {(X' \otimes_{\mathcal{B}} A \# \mathcal{H})  \otimes_{(A \# \mathcal{H})^e}M'}
         \arrow["{\overline{(f,g)}}", from=1-1, to=1-2]
         \arrow["{\underline{(f,g)}}", from=2-1, to=2-2]
         \arrow["{\gamma_{(X,M)}}"', from=1-1, to=2-1]
         \arrow["{\gamma_{(X',M')}}"', from=1-2, to=2-2]
     \end{tikzcd}\]
     where, $\overline{(f,g)}=f \otimes_{\mathcal{H}_{par}}(1_A \otimes_{A^e} g)$ and $\underline{(f,g)}=(f \otimes_{\mathcal{B}}1_A\# 1_{\mathcal{H}})\otimes_{(A \# \mathcal{H})^e}g$. Indeed,

     \begin{align*}
         \gamma_{(X',M')} \left( \overline{(f,g)}(x \otimes_{\mathcal{H}_{par}}(a \otimes_{A^e} m)) \right)
         &=\gamma_{(X,M)} \left(f(x) \otimes_{\mathcal{H}_{par}}(a \otimes_{A^e} g(m))\right) \\
         &= (f(x) \otimes_{\mathcal{B}} 1_A \# 1_{\mathcal{H}}) \otimes_{(A \# \mathcal{H})^e} a \cdot g(m) \\
         &= (f(x) \otimes_{\mathcal{B}} 1_A \# 1_{\mathcal{H}}) \otimes_{(A \# \mathcal{H})^e} g(a \cdot m) \\ 
         &= \underline{(f,g)}\left( (x \otimes_{\mathcal{B}} 1_A \# 1_{\mathcal{H}}) \otimes_{(A \# \mathcal{H})^e} a \cdot m \right) \\
         &= \underline{(f,g)}\left(\gamma_{(X,M)}(x \otimes_{\mathcal{H}_{par}}(a \otimes_{A^e} m))  \right).
     \end{align*}
 Thus, $\gamma$ is a natural isomorphism.
 \end{proof}
 
 \begin{lemma} \label{l_AsmashHisoBAsmashH}
     The isomorphism of $K$-modules
     \begin{align*}
         \phi: A \# \mathcal{H}  &\to \mathcal{B} \otimes_{\mathcal{B}}A \# \mathcal{H} \\ 
                 a \# h          &\mapsto 1_{\mathcal{B}} \otimes_{\mathcal{B}} a \# h
     \end{align*}
     is an isomorphism of $A \# \mathcal{H}$-bimodules.
 \end{lemma}
 \begin{proof}
      For any $b \# k, a \# h, c\# t \in A \# \mathcal{H}$ we have that
     \begin{align*}
         \phi((b \# k)(a \# h)(c \# t))  &= 1_{\mathcal{B}} \otimes_{\mathcal{B}} (b \# k)(a \# h)(c \# t) \\  
         &= \big(1_{\mathcal{B}} \otimes_{\mathcal{B}} (b \# k)(a \# h)\big) \cdot (c \# t) \\ 
         &= 1_{\mathcal{B}} \otimes_{\mathcal{B}} ((k_{(1)} \cdot 1_A)b \# k_{(2)})(a \# h)) \cdot (c \# t) \\ 
        (\text{by } \eqref{e_ehbysmash}) &= 1_{\mathcal{B}} \otimes_{\mathcal{B}} e_{k_{(1)}} \triangleright ((b \# k_{(2)})(a \# h))\big) \cdot (c \# t) \\ 
         &= 1_{\mathcal{B}} \triangleleft e_{k_{(1)}}\otimes_{\mathcal{B}} (b \# k_{(2)})(a \# h)\big) \cdot (c \# t) \\ 
         &= e_{k_{(1)}}1_{\mathcal{B}}e_{k_{(2)}}\otimes_{\mathcal{B}} (b \# k_{(3)})(a \# h)\big) \cdot (c \# t) \\ 
         &= e_{k_{(1)}}\otimes_{\mathcal{B}} (b \# k_{(2)})(a \# h)\big) \cdot (c \# t) \\ 
         &= \big([k_{(1)}]1_{\mathcal{B}}[S(k_{(2)})]\otimes_{\mathcal{B}} (b \# k_{(3)})(a \# h)\big) \cdot (c \# t) \\ 
         &= \big(1_{\mathcal{B}} \triangleleft [S(k_{(1)})]\otimes_{\mathcal{B}} (b \# k_{(2)})(a \# h)\big) \cdot (c \# t) \\ 
         &= (b \# k) \cdot \big(1_{\mathcal{B}} \otimes_{\mathcal{B}} (a \# h)\big) \cdot (c \# t) \\
         &= (b \# k) \cdot \phi (a \# h) \cdot (c \# t).
     \end{align*}
 \end{proof}
 
 \begin{corollary} \label{c_F2F1F}
     The functors $F_2F_1$ and $F$ are naturally isomorphic.
 \end{corollary}
 
 \begin{proof}
     Using Proposition \ref{p_inXB} for the particular case $X=\mathcal{B}$ we obtain that the functors
     \[
         \mathcal{B}\otimes_{\mathcal{H}_{par}} (A \otimes_{A^e}-): (A \# \mathcal{H})^e\!\operatorname{-Mod}  \to K\operatorname{-Mod} 
     \]
     and
     \[
         (\mathcal{B} \otimes_{\mathcal{B}} A \# \mathcal{H}) \otimes_{(A \# \mathcal{H})^e}-:(A \# \mathcal{H})^e\!\operatorname{-Mod}  \to K\operatorname{-Mod}
     \] 
     are naturally isomorphic. Clearly, $F_2F_1 = \mathcal{B}\otimes_{\mathcal{H}_{par}} (A \otimes_{A^e}-)$. On the other hand, by Lemma \ref{l_AsmashHisoBAsmashH} we know that $\mathcal{B} \otimes_{\mathcal{B}} A \# \mathcal{H} \cong A \# \mathcal{H}$ as $A \# \mathcal{H}$-bimodules. Therefore, $F \cong (\mathcal{B} \otimes_{\mathcal{B}} A \# \mathcal{H}) \otimes_{(A \# \mathcal{H})^e}-$, whence we get the desired conclusion.
 \end{proof}

 \underline{In all what follows in this section}, we assume that the cocommutative Hopf algebra $\mathcal{H}$ is projective over $K$.
 
 \begin{lemma} \label{l_HparIsBprojective}
    $\mathcal{H}_{par}$ is projective as a left $\mathcal{B}$-module.
 \end{lemma}
 
 \begin{proof}
  It was proved  in \cite[Theorem 4.8]{ALVES2015137}  that the partial action $(\ref{e_LeftActionB})$ of $\mathcal{H}$ on $\mathcal{B}$ is such that $\hat{\pi}:\mathcal{H}_{par} \to \mathcal{B} \# \mathcal{H}$ is an isomorphism of algebras induced by the partial representation 
     \begin{align*}
         \pi : \mathcal{H}   &\to \mathcal{B} \# \mathcal{H} \\ 
                 h           &\mapsto  1_{\mathcal{B}} \# h.
     \end{align*}
     Recall that the structure of the left $\mathcal{B}$-module $\mathcal{H}_{par}$ is just the induced by the natural inclusion $\mathcal{B} \hookrightarrow \mathcal{H}_{par}$. On the other hand, the structure of $B \# \mathcal{H}$ as a left $\mathcal{B}$-module is determined by the morphism of algebras $\phi_0: \mathcal{B} \to \mathcal{B} \# \mathcal{H}$ such that $w \mapsto w \# 1_{\mathcal{H}}$. Let $h \in \mathcal{H}$ and $x \in \mathcal{H}_{par}$. Then,
     \begin{align*}
         \hat{\pi}(e_h x) 
         &= \hat{\pi}(e_h) \hat{\pi}(x) \\ 
         &= (1 \# h_{(1)})(1 \# S(h_{(2)})) \hat{\pi}(x) \\ 
         &= \bigl( h_{(1)} \cdot 1_{\mathcal{B}} \# h_{(2)}S(h_{(3)}) \bigr)  \hat{\pi}(x) \\ 
         &= \bigl( h \cdot 1_{\mathcal{B}} \# 1_{\mathcal{H}} \bigr)  \hat{\pi}(x) \\ 
         &= \bigl( e_h \# 1_{\mathcal{H}} \bigr)  \hat{\pi}(x) \\ 
         &= e_h \cdot  \hat{\pi}(x).
     \end{align*}
     Therefore, $\hat{\pi}$ is an isomorphism of left $\mathcal{B}$-modules. Therefore, it is enough to prove that $\mathcal{B} \# \mathcal{H}$ is projective as a left $\mathcal{B}$-module. Using the Lemma \ref{lemma:tensorProjectives} for $X=R=\mathcal{B}$, $Y= \mathcal{H}$ and $S = K$ we obtain that $\mathcal{B} \otimes \mathcal{H}$ is projective as a left $\mathcal{B}$-module. Consider the following maps of algebras
 
     \begin{minipage}{.45\textwidth}
         \begin{align*}
             \iota: \mathcal{B} \# \mathcal{H} &\to \mathcal{B} \otimes \mathcal{H}  \\
              z &\mapsto z,
         \end{align*}
     \end{minipage}
     \begin{minipage}{.45\textwidth}
         \begin{align*}
             \#: \mathcal{B} \otimes \mathcal{H}  &\to \mathcal{B} \# \mathcal{H}  \\
              x &\mapsto x(1_A \otimes 1_{\mathcal{H}}).
         \end{align*}
     \end{minipage}% 
 
     \smallskip
 
 \noindent Furthermore, the above maps are morphisms of left $\mathcal{B}$-modules such that $\# \circ \iota = id_{\mathcal{B} \# \mathcal{H}}$. Thus, $\mathcal{B} \# \mathcal{H}$ is a direct summand of $\mathcal{B} \otimes \mathcal{H}$, whence $\mathcal{B} \# \mathcal{H}$ is projective as a left $\mathcal{B}$-module.
 \end{proof}
 
 By Lemma \ref{l_HparIsBprojective} and Proposition \ref{p_DMSA} we obtain the following proposition
 
 \begin{proposition} \label{p_ProjHparmodulesAreBproj}
     Any projective left $\mathcal{H}_{par}$-module is projective as a left $\mathcal{B}$-module.
 \end{proposition}
 
 \begin{lemma} \label{l_BprojectiveResolution}
     Any projective right $\mathcal{H}_{par}$-module is a projective right $\mathcal{B}$-module. Thus, any projective resolution of $\mathcal{B}$ in \textbf{Mod}-$\mathcal{H}_{par}$ is a projective resolution in \textbf{Mod}-$\mathcal{{B}}$.
 \end{lemma}
 
 \begin{proof}
     Let $X$ be a projective right $\mathcal{H}_{par}$-module. Then, by Lemma \ref{l_HparMODHpar} we have that $X$ is a projective left $\mathcal{H}_{par}$-module with action
     \[
         z \triangleright x := x \triangleleft \mathcal{S}(z),\, \forall z \in \mathcal{H}_{par}.
     \]
     Observe that by the cocommutativity of $\mathcal{H}$ we have that
     \[
         \mathcal{S}(e_h) = \mathcal{S}([h_{(1)}][S(h_{(2)})]) =[S^2(h_{(2)})][S([h_{(1)}])] = [h_{(1)}][S([h_{(2)}])] = e_h,
     \]
     and since $\mathcal{B}$ is commutative we conclude that $\mathcal{S}|_{\mathcal{B}}=id_{\mathcal{B}}$. Now by Proposition \ref{p_ProjHparmodulesAreBproj} we have that $X$ is projective as a left $\mathcal{B}$-module with action
     \[
         w \triangleright x := x \triangleleft \mathcal{S}(w) =  x \triangleleft w.
     \]
     Using again that $\mathcal{B}$ is commutative we have that any right $\mathcal{B}$-module is a left $\mathcal{B}$-module with the natural action. Thus, $X$ is projective as right $\mathcal{B}$-module.
 \end{proof}
 
 \begin{proposition} \label{p_F1sendPtoF2acylic}
     $F_1$ sends projective $(A \# \mathcal{H})$-bimodules to left $F_2$-acyclic modules.
 \end{proposition}
 \begin{proof}
     We have to see that
     \[
         (L_nF_2)(A \otimes_{A^e}P) = 0, \, \, \forall n \geq 1,
     \]
     for any projective $(A \# \mathcal{H})^e$-module $P$. Now observe that
     \begin{align*}
         (L_nF_2)(A \otimes_{A^e}P)
         &= L_n(\mathcal{B} \otimes_{\mathcal{H}_{par}} -)(A \otimes_{A^e}P) \\
         &\cong L_n(- \otimes_{\mathcal{H}_{par}} (A \otimes_{A^e}P))(\mathcal{B}) \\
        (\text{by Proposition \ref{p_inXB}}) &\cong L_n((- \otimes_{\mathcal{B}} A\# \mathcal{H}) \otimes_{(A \# \mathcal{H})^e}P)(\mathcal{B}).
     \end{align*}
     Let $Q_{\bullet} \to \mathcal{B}$ a projective resolution of $\mathcal{B}$ in \textbf{Mod}-$\mathcal{H}_{par}$. Then,
     \[
         (L_nF_2)(A \otimes_{A^e}P) \cong L_n((- \otimes_{\mathcal{B}} A\# \mathcal{H}) \otimes_{(A \# \mathcal{H})^e}P)(\mathcal{B}) = H_n \left(  (Q_{\bullet}\otimes_{\mathcal{B}} A\# \mathcal{H}) \otimes_{(A \# \mathcal{H})^e}P \right).
     \]
     Observe that if the complex $(Q_{\bullet}\otimes_{\mathcal{B}} A\# \mathcal{H})$ is exact for all $n \geq 1$ then the complex $(Q_{\bullet}\otimes_{\mathcal{B}} A\# \mathcal{H}) \otimes_{(A \# \mathcal{H})^e}P$ is exact for all $n \geq 1$ since $P$ is projective as $(A \# \mathcal{H})^e$-module, and so
     \[
         H_n \left(  (Q_{\bullet}\otimes_{\mathcal{B}} A\# \mathcal{H}) \otimes_{(A \# \mathcal{H})^e}P \right) = 0, \forall n \geq 1,
     \]
     which is exactly what we want. Therefore, it is enough to show that $(Q_{\bullet}\otimes_{\mathcal{B}} A\# \mathcal{H})$ is exact for all $n \geq 1$. Recall that $(Q_{\bullet}\otimes_{\mathcal{B}} A\# \mathcal{H})$ is exact in $n$ if, and only if,
    \[
         H_n(Q_{\bullet}\otimes_{\mathcal{B}} A\# \mathcal{H})=0.
    \] 
     Notice that by Lemma \ref{l_BprojectiveResolution} we have that $Q_{\bullet} \to \mathcal{B}$ is also a projective resolution of $\mathcal{B}$ in \textbf{Mod}-$\mathcal{B}$, therefore
     \[
         H_n(Q_{\bullet}\otimes_{\mathcal{B}} A\# \mathcal{H})= \operatorname{Tor}_n^{\mathcal{B}}(\mathcal{B}, A \# \mathcal{H})= \left\{\begin{matrix}
             0 & \text{ if } n \geq 1 \\ 
             A \# \mathcal{H} & \text{ if } n=0.
             \end{matrix}\right.
     \]
 \end{proof}
 
 \begin{theorem} \label{t_SSHH}
    Let $\mathcal{H}$ be a cocommutative Hopf $K$-algebra, such that $\mathcal{H}$ is projective as a $K$-module. If $\cdot:\mathcal{H} \otimes A \to A$ is a symmetric partial action of $\mathcal{H}$ on a unital algebra $A,$ then for any $(A \# \mathcal{H})^e$-module $M$ there exist a first quadrant homological spectral sequence $E^r$ such that
     \[
         E^2_{p,q} = \operatorname{Tor}_p^{\mathcal{H}_{par}}(\mathcal{B},H_q(A,M) ) \Rightarrow H_{p+q}(A \#\mathcal{H}, M). 
     \]
 \end{theorem}
 
 \begin{proof}
     Keeping in mind Remark~\ref{r_projectiveAskewGmodulesAreAprojectives} we know that 
     \[ 
        L_q F_1(-)=H_q(A,-),\,\, L_p F_2 (-) = \operatorname{Tor}^{\mathcal{H}_{par}}_p(\mathcal{B},-)\text{ and }L_{p+q}F(-)=H_{p+q}(A \# \mathcal{H}, -). 
     \]
    We also have that $F_1$ and $F_2$ are right exact functors, by Proposition \ref{p_F1sendPtoF2acylic} the functor $F_1$ sends projective $(A \# \mathcal{H})$-bimodules to $F_2$-acyclic modules and by Corollary \ref{c_F2F1F} we have that $F_2 F_1 \cong F$. Thus, by \cite[Theorem 10.48]{RotmanAnInToHoAl} we obtain the desired spectral sequence.
 \end{proof}
 
 \begin{example}\label{ex:separable}
    If $A$ is a separable algebra, then 
    \[
         H_q(A,M) = \left\{\begin{matrix}
             0 & \text{ if } q \geq 1 \\ 
             M/[A,M] & \text{ if } q=0.
             \end{matrix}\right.
     \]
    Therefore, the spectral sequence in Theorem \ref{t_SSHH} collapses on the $p$-axis, and thus we obtain the following isomorphism:
    \[
        H_{n}(A \#\mathcal{H}, M) \cong \operatorname{Tor}_n^{\mathcal{H}_{par}}(\mathcal{B}, M/[A,M]).
    \]
 \end{example}

 \begin{example}\label{ex:group}
    Let $G$ be a group. If $\mathcal{H}=KG$ then the spectral sequence of Theorem \ref{t_SSHC} takes the form
    \[
         E^2_{p,q} = H_p^{par}(G,H_q(A,M) ) \Rightarrow H_{p+q}(A \rtimes G, M), 
     \]
     where $H^{par}_\bullet(G,-):= \operatorname{Tor}_\bullet^{K_{par}G}(B, -)$.
 \end{example}

 \begin{example}
    From \cite{ALVES2015137} we know that there exists a partial action of $\mathcal{H}$ on $\mathcal{B}$ such that $\mathcal{H}_{par} \cong \mathcal{B} \# \mathcal{H}$. Therefore, 
    \[
         E^2_{p,q} = \operatorname{Tor}_p^{\mathcal{H}_{par}}(\mathcal{B},H_q(\mathcal{B},M) ) \Rightarrow H_{p+q}(\mathcal{H}_{par}, M).
     \]
    In particular if $\mathcal{H}=KG$ then the spectral sequence takes the form
    \begin{equation*}
        E_{p,q}^2=H_p^{par}(G, H_q(\mathcal{B},M)) \Rightarrow H_{p+q}(K_{par}G,M).
    \end{equation*}
    Observe that $\mathcal{B}^e$ is generated as a $K$-algebra by the set of idempotent $\{e_g \otimes e_h : g,h \in G \}$. Therefore, $\mathcal{B}^e$ is a Von Neumann regular algebra, and consequently, $\mathcal{B}$ is flat as a $\mathcal{B}^e$-module. Thus, the above spectral sequence collapses on the $p$-axis, and we obtain the following isomorphism
    \[
        H_n^{par}(G, M/[\mathcal{B},M]) \cong H_{n}(K_{par}G,M).
    \]
    The above isomorphism generalizes the Mac Lane isomorphism (see for example \cite[7.4.2]{loday2013cyclic}) in the sense that if $M$ is a $G$-group. Then, 
    \[
        H_{n}(KG,M) = H_{\bullet}(K_{par}G,M) \cong H_\bullet^{par}(G, M) =  H_\bullet(G, M),
    \]
    as it is shown in \cite{MDEJ}.
 \end{example}

%%%%%%%%%%%%%%%%

\section{Cohomology of the partial smash product}\label{sec:cohomology}
Now we proceed to show the existence of a dual cohomological spectral sequence for the Hochschild cohomology. \underline{Recall that} $\mathcal{H}$ is a cocommutative Hopf $K$-algebra, $A$ a unital $K$-algebra and 
 \begin{align*}
     \cdot : \mathcal{H} \otimes A &\to A \\
             h \otimes a &\mapsto h \cdot a
 \end{align*}
 a symmetric partial action of $\mathcal{H}$ on $A$.
 
\begin{proposition} \label{p_HomHparStructure}
    Let $M$ be an $A \# \mathcal{H}$-bimodule. Then, $\operatorname{Hom}_{A^e}(A,M)$ is an $\mathcal{H}_{par}$-module, with the action determined by
    \begin{equation} \label{eq: definition Hpar structure of Hom}
        ([h] \triangleright f)(a) = [h_{(1)}] \triangleright f([S(h_{(2)})]\triangleright a).
    \end{equation}
\end{proposition}
\begin{proof}
    For $h \in \mathcal{H}$ define $\pi_h: \operatorname{Hom}_{A^e}(A,M) \to \operatorname{Hom}_{A^e}(A,M)$, by $\pi_h(f)(a):= [h_{(1)}] \triangleright f([S(h_{(2)})]\triangleright a)$, for all $f \in \operatorname{Hom}_{A^e}(A,M)$ and $a \in A$. Then,
    \begin{align*}
        \pi: \mathcal{H}    &\to \operatorname{End}_K(\operatorname{Hom}_{A^e}(A,M)) \\
                h           &\mapsto \pi_h
    \end{align*}
    is a partial representation. Indeed, for any $h,t \in \mathcal{H}$ we have
    \begin{align*}
        \left(\pi_t \pi_{h_{(1)}}\pi_{S(h_{(2)})}(f) \right)(a)
        &=[t_{(1)}] \triangleright \left(\left(\pi_{h_{(1)}}\pi_{S(h_{(2)})}(f) \right)([S(t_{(2)})]\triangleright a)\right) \\
        &= [t_{(1)}][h_{(1)}] \triangleright \left(\left(\pi_{S(h_{(3)})}(f) \right)([S(h_{(2)})][S(t_{(2)})]\triangleright a)\right) \\
        &= [t_{(1)}][h_{(1)}][S(h_{(3)})] \triangleright \left(f([h_{(4)}][S(h_{(2)})][S(t_{(2)})]\triangleright a)\right) \\
        &= [t_{(1)}][h_{(1)}][S(h_{(2)})] \triangleright \left(f([h_{(3)}][S(h_{(4)})][S(t_{(2)})]\triangleright a)\right) \\
        &= [t_{(1)}h_{(1)}][S(h_{(2)})] \triangleright \left(f([h_{(3)}][S(t_{(2)}h_{(4)})]\triangleright a)\right) \\
        &= [t_{(1)}h_{(1)}] \triangleright \left(\pi_{S(h_{(2)})}(f)([S(t_{(2)}h_{(3)})]\triangleright a)\right) \\
        &= [t_{(1)}h_{(1)}] \triangleright \left(\pi_{S(h_{(3)})}(f)([S(t_{(2)}h_{(2)})]\triangleright a)\right) \\
        &= \pi_{t h_{(1)}}\pi_{S(h_{(2)})}(f)(a).
    \end{align*}
    Analogously, we have that
    \[
        \pi_{S(h_{(1)})}\pi_{h_{(2)}}\pi_t(f)(a) = \pi_{h_{(1)}}\pi_{S(h_{(2)})t}(f)(a).
    \]
    It is clear that $\pi_{1_{\mathcal{H}}}(f)=f$. Then, $\pi$ is a partial representation and thus $\operatorname{Hom}_{A^e}(A,M)$ is a $\mathcal{H}_{par}$-module.
\end{proof}

Now we can define the following functors, by Proposition \ref{p_HomHparStructure}
\begin{equation}
    G_1:= \operatorname{Hom}_{A^e}(A,-): (A \# \mathcal{H})^e \text{-\textbf{Mod}} \to \mathcal{H}_{par}\text{-\textbf{Mod}}.
\end{equation}
Recall that $\mathcal{B}$ is a left $\mathcal{H}_{par}$-module with the action given by (\ref{e_LeftActionB}), thus we can define
\begin{equation}
    G_2:= \operatorname{Hom}_{\mathcal{H}_{par}}(\mathcal{B},-): \mathcal{H}_{par} \text{-\textbf{Mod}} \to K\text{-\textbf{Mod}}.
\end{equation}
The functor used to compute the Hochschild cohomology of $A \# \mathcal{H}$ with coefficients in $M$ is 
\begin{equation}
    G:= \operatorname{Hom}_{(A \# \mathcal{H})^e}(A \# \mathcal{H},-): (A \# \mathcal{H})^e \text{-\textbf{Mod}} \to K\text{-\textbf{Mod}}.
\end{equation}

Recall that by Lemma \ref{l_HparMODHpar} any left $\mathcal{H}_{par}$-module $X$ is a right $\mathcal{H}_{par}$-module. If $X$ is a left $\mathcal{H}_{par}$-module, then by Proposition \ref{p_ArtG-Mod-XtensorArtG} we have that $X \otimes_{\mathcal{B}} A \# \mathcal{H}$ is a $A \# \mathcal{H}$-bimodule with actions:
\begin{equation}
    a \# h \cdot (x \otimes_{\mathcal{B}} c \# t) := x \cdot [S(h_{(1)})] \otimes_{\mathcal{B}} (a \# h_{(2)})(c \# t) = [h_{(1)}] \cdot x \otimes_{\mathcal{B}} (a \# h_{(2)})(c \# t)
\end{equation}
and
\begin{equation}
    (x \otimes_{\mathcal{B}} c \# t) \cdot a \# h  := x  \otimes_{\mathcal{B}}(c \# t)(a \# h ).
\end{equation}

For the cohomological setting we have the dual version of Remark~\ref{r_projectiveAskewGmodulesAreAprojectives}.

\begin{remark} \label{r_Hochschild_cohomology}
   Recall that the morphism of rings $A^e \to \Lambda^e$ induced by the natural inclusion of $A$ into $\Lambda$ determines the structure of $A^e$-module of $\Lambda^e$, and by Lemma~\ref{l_AskewP}, if $\mathcal{H}$ is projective over $K$, then $\Lambda^e$ is projective as $A^e$-module. Thus, by \cite[Corollary 3.6A]{Lam1998LecturesOM} we conclude that any injective $\Lambda^e$-module is an injective $A^e$-module. Consequently,
    \[
        H^\bullet(A,M) \cong R^\bullet G_1(M),
    \]
    for any $\Lambda^e$-module $M$.
\end{remark}

\begin{proposition} \label{p_bifunctorscohomologyisomorphism}
    Let $M$ be a fixed $A \# \mathcal{H}$-bimodule. Then, the functors
    \[
        \operatorname{Hom}_{\mathcal{H}_{par}}(-,\operatorname{Hom}_{A^e}(A,M)): \mathcal{H}_{par}\text{-\textbf{Mod}} \to K\text{-\textbf{Mod}}
    \]
    and
    \[
        \operatorname{Hom}_{(A \# \mathcal{H})^e}(- \otimes_{\mathcal{B}} A \# \mathcal{H}, M): \mathcal{H}_{par} \text{-\textbf{Mod}} \to K\text{-\textbf{Mod}}
    \]
    are naturally isomorphic. On the other hand, if $X$ is a fixed left $\mathcal{H}_{par}$-module, then the functors
    \[
        \operatorname{Hom}_{\mathcal{H}_{par}}(X,\operatorname{Hom}_{A^e}(A,-)): (A \# \mathcal{H})^e \text{-\textbf{Mod}} \to K\text{-\textbf{Mod}}
    \]
    and
    \[
        \operatorname{Hom}_{(A \# \mathcal{H})^e}(X \otimes_{\mathcal{B}} A \# \mathcal{H}, -): (A \# \mathcal{H})^e \text{-\textbf{Mod}} \to K\text{-\textbf{Mod}}
    \]
    are naturally isomorphic.
\end{proposition}
\begin{proof}
    Let $X$ be a left $\mathcal{H}_{par}$-module and $M$ be an $(A \# \mathcal{H})$-bimodule. Define
    \begin{align*}
        \gamma_{(X,M)}: \operatorname{Hom}_{\mathcal{H}_{par}}(X,\operatorname{Hom}_{A^e}(A,M)) &\to \operatorname{Hom}_{(A \# \mathcal{H})^e}(X \otimes_{\mathcal{B}} A \# \mathcal{H}, M) \\ 
        f &\mapsto \gamma_{(X,M)}(f),
    \end{align*}
    such that 
    \[
        \gamma_{(X,M)}(f)(x \otimes_{\mathcal{B}} (a \# h)) = f_x(1_A) \cdot (a \# h),
    \]
    where $f \in \operatorname{Hom}_{\mathcal{H}_{par}}(X,\operatorname{Hom}_{A^e}(A,M))$ and $f_x := f(x)$. First, we have to verify that $\gamma_{(X, M)}(f)$ is well-defined. Indeed, notice that
    \begin{align*}
        \gamma_{(X,M)}&(f)(x \cdot e_h \otimes_{\mathcal{B}} (a \# h)) \\
        &=  f_{x \cdot e_h}(1_A) \cdot (a \# h) \\
        &=  f_{e_h \cdot x}(1_A) \cdot (a \# h) \\
        &=  (e_h \triangleright f_{x})(1_A) \cdot (a \# h) \\
        &=  ([h_{(1)}][S(h_{(2)})] \triangleright f_{x})(1_A) \cdot (a \# h) \\
        &=  \Big([h_{(1)}] \triangleright \big([S(h_{(2)})] \triangleright f_{x} \big) \Big)(1_A) \cdot (a \# h) \\
        (\text{by }\eqref{eq: definition Hpar structure of Hom}) &=  [h_{(1)}] \triangleright \Bigl(\big([S(h_{(3)})] \triangleright f_{x}\big)([S(h_{(2)})] \triangleright 1_A)\Bigr) \cdot (a \# h) \\
        &=  [h_{(1)}][S(h_{(2)})] \triangleright f_{x} \big([h_{(3)}][S(h_{(4)})] \triangleright 1_A \big) \cdot (a \# h) \\
        (\flat) &=  \left(  e_{h_{(1)}} \triangleright f_{x}(h_{(2)} \cdot 1_A) \right) \cdot (a \# h) \\
        &=  \left( (1_A \# h_{(1)})(1_A \# S(h_{(2)})) \cdot f_{x}(h_{(3)} \cdot 1_A) \cdot (1_A \# h_{(4)})(1_A \# S(h_{(5)})) \right) \cdot (a \# h) \\
        &=  \left((h_{(1)}\cdot 1_A \# h_{(2)}S(h_{(3)})) \cdot f_{x}(h_{(4)} \cdot 1_A) \cdot (h_{(5)}\cdot 1_A \# h_{(6)}S(h_{(7)})) \right) \cdot (a \# h) \\
        &=  \left((h_{(1)}\cdot 1_A \# 1_{\mathcal{H}}) \cdot f_{x}(h_{(2)} \cdot 1_A) \cdot (h_{(3)}\cdot 1_A \# 1_{\mathcal{H}}) \right) \cdot (a \# h) \\
        &=  f_{x}\left( (h_{(1)} \cdot 1_A)(h_{(2)} \cdot 1_A)(h_{(3)} \cdot 1_A)\right) \cdot (a \# h) \\
        &=  f_{x}(h \cdot 1_A) \cdot (a \# h) \\
        &=  f_{x}(1_A) \cdot (h\cdot 1_A \# 1_{\mathcal{H}})(a \# h) \\
        &=  f_{x}(1_A) \cdot ((h\cdot 1_A) a \# h) \\
    (\flat \flat)    &=  f_{x}(1_A) \cdot (e_h \triangleright a \# h) \\
        &= \gamma_{(X,M)}(f)(x \otimes_{\mathcal{B}} e_h \triangleright (a \# h)),
    \end{align*}
    where the equalities $(\flat)$ and $(\flat \flat)$ are due Lemma \ref{l_teneq} $(ii)$.
    Now we have to verify that $\gamma_{(X,M)}(f)$ is a morphism of $(A \# \mathcal{H})$-bimodules. It is clear that $\gamma_{(X, M)}(f)$ is a morphism of right $A \# \mathcal{H}$-modules, so we only have to see that is a morphism of left $A \# \mathcal{H}$-modules. Observe that
    \begin{align*}
        \gamma_{(X,M)}(f)(b \# t \cdot (x\otimes_{\mathcal{B}} (a \# h)))
        &= \gamma_{(X,M)}(f)([t_{(1)}] \cdot x \otimes_{\mathcal{B}} (b \# t_{(2)})(a \# h)) \\
        &= f_{[t_{(1)}] \cdot x}(1_A) \cdot (b \# t_{(2)})(a \# h) \\
        &= ([t_{(1)}] \triangleright f_{x})(1_A) \cdot (b \# t_{(2)})(a \# h) \\
        &= \left([t_{(1)}] \triangleright f_{x}([S(t_{(2)})] \triangleright 1_A)\right) \cdot (b \# t_{(3)})(a \# h) \\
        &= (1_A \# t_{(1)}) \cdot f_{x}(S(t_{(3)}) \cdot1_A) \cdot (1_A \# S(t_{(2)}))(b \# t_{(4)})(a \# h) \\
        &= (1_A \# t_{(1)}) \cdot f_{x}(S(t_{(5)}) \cdot1_A) \cdot (S(t_{(2)})\cdot b \# S(t_{(3)})t_{(4)})(a \# h) \\
        &= (1_A \# t_{(1)}) \cdot f_{x}(S(t_{(4)}) \cdot1_A) \cdot (S(t_{(2)})\cdot b \# \varepsilon(t_{(3)})1_{\mathcal{H}})(a \# h) \\
        &= (1_A \# t_{(1)}) \cdot f_{x}(S(t_{(3)}) \cdot1_A) \cdot (S(t_{(2)})\cdot b \#1_{\mathcal{H}})(a \# h) \\
        &= (1_A \# t_{(1)}) \cdot f_{x}\big((S(t_{(3)}) \cdot1_A)(S(t_{(2)})\cdot b)\big) \cdot (a \# h) \\
        &= (1_A \# t_{(1)}) \cdot f_{x}\big((S(t_{(2)})\cdot b)\big) \cdot (a \# h) \\
        &= (1_A \# t_{(1)})((S(t_{(2)})\cdot b) \# 1_{\mathcal{H}}) \cdot f_{x}(1_A) \cdot (a \# h) \\
       (\text{by Lemma }\ref{l_commutSmash}) &= (b \# t) \cdot f_{x}(1_A) \cdot (a \# h) \\
       &= (b \# t) \cdot \gamma_{(X,M)}(f)(x\otimes_{\mathcal{B}} (a \# h)).
    \end{align*}
    To show that $\gamma_{(X,M)}$ is an isomorphism of $K$-modules for any $X \in \mathcal{H}_{par}$-\textbf{Mod} and $M \in (A \# \mathcal{H})^e$-\textbf{Mod} we will define its inverse map
    \[
        \Lambda_{(X,M)}:\operatorname{Hom}_{(A \# \mathcal{H})^e}(X \otimes_{\mathcal{B}} A \# \mathcal{H}, M) \to \operatorname{Hom}_{\mathcal{H}_{par}}(X,\operatorname{Hom}_{A^e}(A,M))
    \]
    such that
    \[
        \big(\Lambda_{(X,M)}(f) \big)_x(a):= f(x \otimes_{\mathcal{B}} a \# 1_{\mathcal{H}}),
    \]
    for all $f \in \operatorname{Hom}_{(A \# \mathcal{H})^e}(X \otimes_{\mathcal{B}} A \# \mathcal{H}, M)$. Observe that $ \big(\Lambda_{(X,M)}(f) \big)_x$ is a morphism of $A^e$-modules. Indeed,
    \begin{align*}
        \big(\Lambda_{(X,M)}(f) \big)_x(bac)
        &= f(x \otimes_{\mathcal{B}} bac \# 1_{\mathcal{H}}) \\ 
        &= f(x \otimes_{\mathcal{B}} (b \# 1_{\mathcal{H}})(a \# 1_{\mathcal{H}})(c \# 1_{\mathcal{H}})) \\
        &= f((b \# 1_{\mathcal{H}}) \cdot (x \otimes_{\mathcal{B}} (a \# 1_{\mathcal{H}})) \cdot (c \# 1_{\mathcal{H}})) \\
        &= (b \# 1_{\mathcal{H}}) \cdot f( (x \otimes_{\mathcal{B}} (a \# 1_{\mathcal{H}})))\cdot (c \# 1_{\mathcal{H}}) \\
        &= b \cdot \big(\Lambda_{(X,M)}(f) \big)_x(a)\cdot c.
    \end{align*}
    Now observe that $\Lambda_{(X,M)}(f)$ is a morphism of $\mathcal{H}_{par}$-modules.
    \begin{align*}
        \big( \Lambda_{(X,M)}(f) \big)_{[h] \cdot x}(a)
        &= f([h] \cdot x \otimes_{\mathcal{B}} a \# 1_{\mathcal{H}}) \\
        &= f(e_{h_{(1)}}[h_{(2)}] \cdot x \otimes_{\mathcal{B}} a \# 1_{\mathcal{H}}) \\
        &= f(e_{h_{(1)}} \cdot ([h_{(2)}] \cdot x) \otimes_{\mathcal{B}} a \# 1_{\mathcal{H}}) \\
        &= f(([h_{(2)}] \cdot x) \cdot e_{h_{(1)}} \otimes_{\mathcal{B}} a \# 1_{\mathcal{H}}) \\
        &= f([h_{(2)}] \cdot x \otimes_{\mathcal{B}} e_{h_{(1)}} \triangleright (a \# 1_{\mathcal{H}})) \\ 
        &= f([h_{(3)}] \cdot x \otimes_{\mathcal{B}} (1_A \# h_{(1)})(1_A \# S(h_{(2)}))(a \# 1_{\mathcal{H}})) \\ 
        &= f\left( (1_A \# h_{(1)}) \cdot (x \otimes_{\mathcal{B}} (1_A \# S(h_{(2)}))(a \# 1_{\mathcal{H}}))\right)\\ 
        &= f\left( (1_A \# h_{(1)}) \cdot (x \otimes_{\mathcal{B}} (S(h_{(2)}) \cdot a \# S(h_{(3)})))\right)\\ 
        &= f\left( (1_A \# h_{(1)}) \cdot (x \otimes_{\mathcal{B}} (S(h_{(2)}) \cdot a \# 1_{\mathcal{H}})) \cdot ( 1_A \# S(h_{(3)})) \right)\\ 
        &= (1_A \# h_{(1)}) \cdot f\left( (x \otimes_{\mathcal{B}} (S(h_{(3)}) \cdot a \# 1_{\mathcal{H}}))  \right) \cdot ( 1_A \# S(h_{(2)})) \\ 
        &= [h_{(1)}] \triangleright \left(f(x \otimes_{\mathcal{B}} ([S(h_{(2)})]\triangleright a )\# 1_{\mathcal{H}})  \right) \\ 
    (\text{by } \eqref{eq: definition Hpar structure of Hom})    &= [h_{(1)}] \triangleright \big( \Lambda_{(X,M)}(f) \big)_{x}([S(h_{(2)})] \triangleright  a) \\ 
        &= ([h] \triangleright \big( \Lambda_{(X,M)}(f) \big)_{x})(a).
    \end{align*}
    Thus, $\Lambda_{(X,M)}$ is well-defined. Finally, by direct computations we obtain that $\gamma_{(X,M)}$ and $\Lambda_{(X,M)}$ are mutually inverses:
    \begin{align*}
       ( \Lambda_{(X,M)}(\gamma_{(X,M)}(f)))_x(a)
       &= \gamma_{(X,M)}(f)(x \otimes_{\mathcal{B}} a \# 1_{\mathcal{H}}) \\
       &= f_x(1_A)\cdot (a \# 1_{\mathcal{H}}) \\
       &= f_x(a) \\
    \end{align*}
    and
    \begin{align*}
        ( \gamma_{(X,M)}(\Lambda_{(X,M)}(f)))(x \otimes_{\mathcal{B}}(a \# h))
        &= \big( \Lambda_{(X,M)}(f) \big)_{x}(1_A) \cdot a \# h \\ 
        &= f(x \otimes_{\mathcal{B}} 1_A \# 1_{\mathcal{H}}) \cdot a \# h \\ 
        &= f(x \otimes_{\mathcal{B}} a \# h).
    \end{align*}

    For a fixed $M$ define $\gamma_M(X)=\gamma_{(X,M)}$. Then,
    \[
        \gamma_M: \operatorname{Hom}_{\mathcal{H}_{par}}(-, \operatorname{Hom}_{A^e}(A,M)) \to \operatorname{Hom}_{(A \# \mathcal{H})^e}(- \otimes_{\mathcal{B}} A \# \mathcal{H},M)
    \] 
    is a natural transformation. Indeed, let $\xi: X \to X'$ be a map of left $\mathcal{H}_{par}$-modules. Then, we want to verify that the following diagram commutes
    \[\begin{tikzcd}
        {\operatorname{Hom}_{\mathcal{H}_{par}}(X, \operatorname{Hom}_{A^e}(A,M))} & {\operatorname{Hom}_{\mathcal{H}_{par}}(X', \operatorname{Hom}_{A^e}(A,M))} \\
        {\operatorname{Hom}_{(A \# \mathcal{H})^e}(X \otimes_{\mathcal{B}} A \# \mathcal{H},M)} & {\operatorname{Hom}_{(A \# \mathcal{H})^e}(X' \otimes_{\mathcal{B}} A \# \mathcal{H},M)}
        \arrow["{\xi^*}"', from=1-2, to=1-1]
        \arrow["{\xi_*}"', from=2-2, to=2-1]
        \arrow["{\gamma_{(X,M)} \quad}"', from=1-1, to=2-1]
        \arrow["{\gamma_{(X',M)} \quad}"', from=1-2, to=2-2]
    \end{tikzcd}\]
    where
    \[
        \xi^* := \operatorname{Hom}_{\mathcal{H}_{par}}(\xi,\operatorname{Hom}_{A^e}(A,M)),
    \]
    and
    \[
        \xi_* := \operatorname{Hom}_{(A \# \mathcal{H})^e}(\xi \otimes_{\mathcal{B}} A \# \mathcal{H}, M).
    \]
    Indeed,
    \begin{align*}
        (\gamma_{(X,M)}(\xi^*(f)))(x \otimes_{\mathcal{B}} a \# h)
        &= (\xi^*(f))_x(1_A) \cdot a \# h \\
        &= f_{\xi(x)}(1_A) \cdot a \# h
    \end{align*}
    and
    \begin{align*}
        \xi_*( \gamma_{(X',M)}(f))(x \otimes_{\mathcal{B}} a \# h) 
        &= \gamma_{(X',M)}(f)(\xi(x) \otimes_{\mathcal{B}} a \# h) \\ 
        &= f_{\xi(x)}(1_A) \cdot a \# h.
    \end{align*}
    Thus, the above diagram commutes. Finally, for a fixed $\mathcal{H}_{par}$-module $X$ we define
    \[
        \gamma_X: \operatorname{Hom}_{\mathcal{H}_{par}}(X, \operatorname{Hom}_{A^e}(A,-)) \to \operatorname{Hom}_{(A \# \mathcal{H})^e}(X \otimes_{\mathcal{B}} A \# \mathcal{H},-),
    \] 
    such that $\gamma_{X}(M):=\gamma_{(M,X)}$. Analogously we have that $\gamma_X$ is a natural transformation. Indeed, we have to see that for any $\zeta: M \to M'$ morphism of $A \# \mathcal{H}$-bimodules the following diagram commutes
    \[\begin{tikzcd}
        {\operatorname{Hom}_{\mathcal{H}_{par}}(X, \operatorname{Hom}_{A^e}(A,M))} & {\operatorname{Hom}_{\mathcal{H}_{par}}(X, \operatorname{Hom}_{A^e}(A,M'))} \\
        {\operatorname{Hom}_{(A \# \mathcal{H})^e}(X \otimes_{\mathcal{B}} A \# \mathcal{H},M)} & {\operatorname{Hom}_{(A \# \mathcal{H})^e}(X \otimes_{\mathcal{B}} A \# \mathcal{H},M')}
        \arrow["{\gamma_{(X,M)} \quad}"', from=1-1, to=2-1]
        \arrow["{\gamma_{(X,M')} \quad}"', from=1-2, to=2-2]
        \arrow["{\zeta^*}", from=1-1, to=1-2]
        \arrow["{\zeta_*}", from=2-1, to=2-2]
    \end{tikzcd}\]
    where
    \[
        \zeta^* := \operatorname{Hom}_{\mathcal{H}_{par}}(X,\operatorname{Hom}_{A^e}(A,\zeta)),
    \]
    and
    \[
        \zeta_* := \operatorname{Hom}_{(A \# \mathcal{H})^e}(X \otimes_{\mathcal{B}} A \# \mathcal{H}, \zeta).
    \]
    Observe that
    \begin{align*}
        (\gamma_{(X,M')}(\zeta^*(f)))(x \otimes_{\mathcal{B}} a \# h)
        &= (\zeta^*(f))_x (1_A) \cdot (a \# h) \\
        &= \zeta(f_x(1_A)) \cdot (a \# h)
    \end{align*}
    and
    \begin{align*}
        \zeta_*(\gamma_{(X,M)}(f))(x \otimes_{\mathcal{B}} a \# h)
        &= \zeta(\gamma_{(X,M)}(f)(x \otimes_{\mathcal{B}} a \# h)) \\
        &= \zeta \left(f_x(1_A) \cdot (a \# h)\right) \\
        &= \zeta \left(f_x(1_A)\right)\cdot (a \# h).
    \end{align*}
    Thus, the above diagram commutes.
\end{proof}

\begin{corollary} \label{c_G2G1congG}
    The functors $G_2G_1$ and $G$ are naturally isomorphic.
\end{corollary}
\begin{proof}
    By Proposition \ref{p_bifunctorscohomologyisomorphism}, taking $X=\mathcal{B}$ we have that the functors 
    \[
        \operatorname{Hom}_{\mathcal{H}_{par}}(\mathcal{B},\operatorname{Hom}_{A^e}(A,-)): (A \# \mathcal{H})^e \text{-\textbf{Mod}} \to K\text{-\textbf{Mod}}
    \]
    and
    \[
        \operatorname{Hom}_{(A \# \mathcal{H})^e}(\mathcal{B} \otimes_{\mathcal{B}} A \# \mathcal{H}, -): (A \# \mathcal{H})^e \text{-\textbf{Mod}} \to K\text{-\textbf{Mod}}
    \]
    are naturally isomorphic. It is clear that $G_2G_1 = \operatorname{Hom}_{\mathcal{H}_{par}}(\mathcal{B},\operatorname{Hom}_{A^e}(A,-))$. On the other hand by Lemma \ref{l_AsmashHisoBAsmashH} we have that $\mathcal{B} \otimes_{\mathcal{B}} A \# \mathcal{H} \cong A \# \mathcal{H}$ as $A \# \mathcal{H}$-bimodules. Thus, $G \cong \operatorname{Hom}_{(A \# \mathcal{H})^e}(\mathcal{B} \otimes_{\mathcal{B}} A \# \mathcal{H}, -)$.
\end{proof}

\underline{In all what follows}, we assume that $\mathcal{H}$ is projective over $K$.

\begin{proposition}\label{prop:G1sendItoD2acylic}
    $G_1$ sends injective $(A \# \mathcal{H})$-bimodules to right $G_2$-acyclic modules.
\end{proposition}

\begin{proof}
    We have to see that
    \[
        (R^nG_2)(\operatorname{Hom}_{A^e}(A,Q)) = 0 \, \forall n \geq 1,
    \]
    for any injective $(A \# \mathcal{H})^e$-module $Q$. Now observe that
    \begin{align*}
        (R^n G_2)(\operatorname{Hom}_{A^e}(A,Q))
        &= R^n(\operatorname{Hom}_{\mathcal{H}_{par}}(\mathcal{B},-))(\operatorname{Hom}_{A^e}(A,Q)) \\
        &\cong R^n(\operatorname{Hom}_{\mathcal{H}_{par}}(-,\operatorname{Hom}_{A^e}(A,Q)))(\mathcal{B}) \\
       (\text{by Proposition \ref{p_bifunctorscohomologyisomorphism}}) &\cong R^n(\operatorname{Hom}_{(A \# \mathcal{H})^e}(- \otimes_{\mathcal{B}} A\# \mathcal{H}, Q))(\mathcal{B}).
    \end{align*}
    Let $P_{\bullet} \to \mathcal{B}$ a projective resolution of $\mathcal{B}$ in $\mathcal{H}_{par}$-\textbf{Mod}. Then,
    \[
        (R^nG_2)(\operatorname{Hom}_{A^e}(A,Q)) \cong H^n \left( \operatorname{Hom}_{(A \# \mathcal{H})^e}(P_\bullet \otimes_{\mathcal{B}} A\# \mathcal{H}, Q)\right).
    \]
    Observe that if the complex $P_{\bullet}\otimes_{\mathcal{B}} A\# \mathcal{H}$ is exact for all $n \geq 1$ then the chain complex $\operatorname{Hom}_{(A \# \mathcal{H})^e}(P_\bullet \otimes_{\mathcal{B}} A\# \mathcal{H}, Q)$ is exact for all $n \geq 1$ since $Q$ is injective as $(A \# \mathcal{H})^e$-module, and so
    \[
        H^n \left( \operatorname{Hom}_{(A \# \mathcal{H})^e}(P_\bullet \otimes_{\mathcal{B}} A\# \mathcal{H}, Q)\right) = 0, \forall n \geq 1,
    \]
    which is exactly what we want. Therefore, it is enough to show that $P_{\bullet}\otimes_{\mathcal{B}} A\# \mathcal{H}$ is exact for all $n \geq 1$, recall that $P_{\bullet}\otimes_{\mathcal{B}} A\# \mathcal{H}$ is exact in $n$ if, and only if,
   \[
        H_n(P_{\bullet}\otimes_{\mathcal{B}} A\# \mathcal{H})=0.
   \] 
    Notice that by Proposition \ref{p_ProjHparmodulesAreBproj} we have that $P_{\bullet} \to \mathcal{B}$ is also a projective resolution of $\mathcal{B}$ in $\mathcal{B}$-\textbf{Mod}, therefore
    \[
        H_n(P_{\bullet}\otimes_{\mathcal{B}} A\# \mathcal{H})= \operatorname{Tor}_n^{\mathcal{B}}(\mathcal{B}, A \# \mathcal{H})= \left\{\begin{matrix}
            0 & \text{ if } n \geq 1 \\ 
            A \# \mathcal{H} & \text{ if } n=0.
            \end{matrix}\right.
    \]
\end{proof}

\begin{theorem} \label{t_SSHC}
    Let $\mathcal{H}$ be a cocommutative Hopf $K$-algebra, which is projective as a $K$-module, and $\cdot:\mathcal{H} \otimes A \to A$ be a symmetric partial action of $\mathcal{H}$  on $A$. Then, for any $(A \# \mathcal{H})^e$-module $M$ there exist a cohomological spectral sequence $E_r$ such that
    \[
        E_2^{p,q} = \operatorname{Ext}^p_{\mathcal{H}_{par}}(\mathcal{B},H^q(A,M) ) \Rightarrow H^{p+q}(A \#\mathcal{H}, M). 
    \]
\end{theorem}

\begin{proof}
    Recalling Remark~\ref{r_Hochschild_cohomology} we know that 
    \[ 
       R^q G_1(-)=H^q(A,-),\,\, R^p G_2 (-) = \operatorname{Ext}_{\mathcal{H}_{par}}^p(\mathcal{B},-)\text{ and }R^{p+q}G(-)=H^{p+q}(A \# \mathcal{H}, -). 
    \]
   We also have that $G_1$ and $G_2$ are left exact functors, by Proposition \ref{prop:G1sendItoD2acylic} the functor $G_1$ sends injective $(A \# \mathcal{H})$-bimodules to right $G_2$-acyclic modules and by Corollary \ref{c_G2G1congG} we have that $G_2 G_1 \cong G$. Thus, by \cite[Theorem 10.47]{RotmanAnInToHoAl} we obtain the desired spectral sequence.
\end{proof}

%%%%%%%%%%%%%

\section{Hopf (co)homology based on partial representations}\label{sec:Hopf(co)homology}
 In analogy with the case of groups (see \cite{AlAlRePartialCohomology} and \cite{MMAMDDHKPartialHomology}) we give the following definition:
 \begin{definition} \label{d_partial cohomology}
    Let $\mathcal{H}$ be a cocommutative Hopf algebra. For a left $\mathcal{H}_{par}$-modue $M$, we define the partial Hopf homology of $\mathcal{H}$ with coefficients in $M$ by
    \begin{equation}
        H^{par}_\bullet(\mathcal{H}, M) := \operatorname{Tor}_\bullet^{\mathcal{H}_{par}}(\mathcal{B}, M).
    \end{equation} 
    Analogously, we define the partial Hopf cohomology of $\mathcal{H}$ with coefficients in $M$ by
    \begin{equation}
        H_{par}^\bullet(\mathcal{H}, M) := \operatorname{Ext}^\bullet_{\mathcal{H}_{par}}(\mathcal{B}, M).
    \end{equation} 
 \end{definition}

 Observe that as in the case of groups, the above-defined cohomology differs from the cohomology based on partial actions introduced in \cite{EBAMMT}.

 In view of Definition~\ref{d_partial cohomology} we can reformulate Theorem~\ref{t_SSHH} and Theorem~\ref{t_SSHC} as follows:

 \begin{theorem} \label{t_SS}
    Let $\cdot: \mathcal{H} \otimes A \to A$ be a symmetric partial action of a cocommutative Hopf algebra on $A,$ such that
    $\mathcal{H}$ is projective over $K.$ Then for any $(A \# \mathcal{H})^e$-module $M$ there exists a first quadrant homological spectral sequence $E^r$ such that
    \[
        E^2_{p,q} = H^{par}_p(\mathcal{H},H_q(A,M) ) \Rightarrow H_{p+q}(A \#\mathcal{H}, M),
    \]
    and a third quadrant cohomological spectral sequence $E_r$ such that
    \[
        E_2^{p,q} = H^p_{par}(\mathcal{H},H^q(A,M) ) \Rightarrow H^{p+q}(A \#\mathcal{H}, M). 
    \]
\end{theorem}

We proceed by showing that the above-defined (co)homology is a generalization of the usual (co)homology for $\mathcal{H}$-modules.

\begin{remark}
    It is easy to see that the identity map $id_{\mathcal{H}}: \mathcal{H} \to \mathcal{H}$ is a partial representation. Thus, by Theorem~\ref{t_Universal Property Partial Actions} we have a surjective homomorphism of algebras $\mathcal{H}_{par} \to \mathcal{H}$ such that $[h] \mapsto h$. Consequently, any left (right) $\mathcal{H}$-module will be a left (right) $\mathcal{H}_{par}$-module.
\end{remark}

 \begin{lemma} \label{l_PHisHprojective}
     Let $P$ be a projective left (right) $\mathcal{H}_{par}$-module. Then, $\mathcal{H} \otimes_{\mathcal{H}_{par}}P$ ($P\otimes_{\mathcal{H}_{par}} \mathcal{H}$) is projective as left (right) $\mathcal{H}$-module.
 \end{lemma}
 
 \begin{proof}
     Let $P$ be a projective right $\mathcal{H}_{par}$-modules. Then, there exists a right $\mathcal{H}_{par}$-module $Q$ such that $P \oplus Q \cong \bigoplus_{i \in I} \mathcal{H}_{par}$. Therefore,
     \begin{align*}
         (P \otimes_{\mathcal{H}_{par}}\mathcal{H}) \oplus (Q \otimes_{\mathcal{H}_{par}}\mathcal{H})
         &=(P \oplus Q ) \otimes_{\mathcal{H}_{par}}\mathcal{H} \\
         &\cong \left(\bigoplus_{i \in I} \mathcal{H}_{par}\right) \otimes_{\mathcal{H}_{par}}\mathcal{H} \\
         &\cong \bigoplus_{i \in I} (\mathcal{H}_{par} \otimes_{\mathcal{H}_{par}}\mathcal{H}) \\
         &\cong \bigoplus_{i \in I} \mathcal{H}.
     \end{align*}
     Thus, $P \otimes_{\mathcal{H}_{par}}\mathcal{H}$ is projective a right $\mathcal{H}$-module. Analogously one proves that $\mathcal{H} \otimes_{\mathcal{H}_{par}}P$ is projective as left $\mathcal{H}$-module.
 \end{proof}

 Recall that the co-unit $\varepsilon: \mathcal{H} \to K$ is a homomorphism of algebras, thus we have a homomorphism of algebras $\mathcal{H}_{par} \to K$ such that $[h] \mapsto \varepsilon(h)$.
 
 \begin{lemma} \label{l_HoBFunctors}
     The functors $- \otimes_{\mathcal{H}_{par}} \mathcal{H}$ and $- \otimes_{\mathcal{B}}K$ are naturally isomorphic. Analogously, the functors $\mathcal{H} \otimes_{\mathcal{H}_{par}} -$ and $K \otimes_{\mathcal{B}} -$ are naturally isomorphic.
 \end{lemma}
 
 \begin{proof}
     Let $X$ be a right $\mathcal{H}_{par}$-module. 
     \begin{align*}
         \tilde{\eta}_X: X \times \mathcal{H} &\to X \otimes_{\mathcal{B}}K \\ 
                         (x,h)           &\mapsto x \cdot [h] \otimes_{\mathcal{B}} 1.
     \end{align*}
     Notice that $\tilde{f}$ is $\mathcal{H}_{par}$-balanced. Indeed,
     \begin{align*}
         \tilde{\eta}_X(x \cdot [t],h)
         &=  x \cdot [t][h] \otimes_{\mathcal{B}}1 \\
         &=  x \cdot [t][h_{(1)}]e_{S(h_{(2)})} \otimes_{\mathcal{B}}1 \\
         &=  x \cdot [th_{(1)}]e_{S(h_{(2)})} \otimes_{\mathcal{B}}1 \\
         &=  x \cdot [th_{(1)}] \otimes_{\mathcal{B}}e_{S(h_{(2)})} \triangleright 1 \\
         &=  x \cdot [th_{(1)}] \otimes_{\mathcal{B}} \varepsilon(S(h_{(2)})) \\
         &=  x \cdot [th_{(1)}] \otimes_{\mathcal{B}} \varepsilon(h_{(2)}) \\
         &=  x \cdot [th_{(1)}\varepsilon(h_{(2)})] \otimes_{\mathcal{B}} 1 \\
         &=  x \cdot [th] \otimes_{\mathcal{B}} 1 \\
         &=  \tilde{\eta}_X(x,th) = \tilde{\eta}_X(x,[t] \triangleright h).
     \end{align*}
     Therefore, the map
     \begin{align*}
         \eta_X: X \otimes_{\mathcal{H}_{par}} \mathcal{H} &\to X \otimes_{\mathcal{B}}K \\ 
                     x \otimes_{\mathcal{H}_{par}} h  &\mapsto x \cdot [h] \otimes_{\mathcal{B}} 1
     \end{align*}
     is well-defined. Let $f: X \to X'$ be a morphism of $\mathcal{H}_{par}$-modules. Then, the following diagram commutes
     \[\begin{tikzcd}[column sep=large]
         { X \otimes_{\mathcal{H}_{par}} \mathcal{H}} & { X' \otimes_{\mathcal{H}_{par}} \mathcal{H}} \\
         { X \otimes_{\mathcal{B}} K} & {X' \otimes_{\mathcal{B}} K}
         \arrow["{f \otimes_{_{\scriptscriptstyle\mathcal{H}_{par}}} 1_{\mathcal{H}}}", from=1-1, to=1-2]
         \arrow["{f \otimes_{\mathcal{B}} 1}", from=2-1, to=2-2]
         \arrow["{\eta_X}"', from=1-1, to=2-1]
         \arrow["{\eta_{X'}}"', from=1-2, to=2-2]
     \end{tikzcd}\]
     Indeed,
     \begin{align*}
         \eta_{X'} \circ (f \otimes_{{H}_{par}} 1_{\mathcal{H}})(x \otimes_{\mathcal{H}_{par}}h) 
         &= \eta_{X'}(f(x) \otimes_{\mathcal{H}_{par}}h) \\ 
         &= f(x) \cdot [h] \otimes_{\mathcal{B}} 1 \\ 
         &= f(x \cdot [h] )  \otimes_{\mathcal{B}} 1 \\ 
         &=  (f \otimes_{\mathcal{B}} 1)(x \cdot [h]   \otimes_{\mathcal{B}} 1) \\
         &=  (f \otimes_{\mathcal{B}} 1) \circ \eta_X (x \otimes_{\mathcal{H}_{par}} h).
     \end{align*}
     Thus, $\eta: (- \otimes_{\mathcal{H}_{par}} \mathcal{H}) \to (- \otimes_{\mathcal{B}}K)$ is a natural transformation. Finally, consider the map
     \begin{align*}
         \tilde{\Phi}_X: X \times K  &\to X \otimes_{\mathcal{H}_{par}} \mathcal{H} \\ 
                         (x,r)       &\mapsto x \otimes_{\mathcal{H}_{par}} r 1_{\mathcal{H}}.
     \end{align*}
     Observe that $\tilde{\Phi}_X$ is $\mathcal{B}$-balanced. Indeed,
     \begin{align*}
         \tilde{\Phi}_X(x \cdot e_h,r)
         &= x \cdot e_h \otimes_{\mathcal{B}} r 1_{\mathcal{H}} \\
         &= x \otimes_{\mathcal{B}} e_h \triangleright ( r 1_{\mathcal{H}} )\\
         &= x \otimes_{\mathcal{B}} \varepsilon(h) r 1_{\mathcal{H}} \\
         &= x \otimes_{\mathcal{B}} (e_h \triangleright r) 1_{\mathcal{H}} \\
         &=  \tilde{\Phi}_X(x, e_h \triangleright r).
     \end{align*}
     Thus, the map
     \begin{align*}
         \Phi_X: X \otimes_{\mathcal{B}} K  &\to X \otimes_{\mathcal{H}_{par}} \mathcal{H} \\ 
                         x \otimes_{\mathcal{B}} r       &\mapsto x \otimes_{\mathcal{H}_{par}} r 1_{\mathcal{H}}
     \end{align*}
     is well-defined. The maps $\eta_X$ and $\Phi_X$ are mutual inverses:
     \[
         \eta_X \circ \Phi_X (x \otimes_{\mathcal{B}} r) = \eta_X(x \otimes_{\mathcal{H}_{par}} r 1_{\mathcal{H}}) = x \cdot (r1_{\mathcal{H}}) \otimes_{\mathcal{B}} 1 = x \otimes_{\mathcal{B}}  r,
     \]
     and
     \[
         \Phi_X \circ \eta_X (x \otimes_{\mathcal{H}_{par}} h) = \Phi_X(x \cdot [h] \otimes_{\mathcal{B}} 1) = x \cdot [h] \otimes_{\mathcal{H}_{par}} 1_{\mathcal{H}} = x \otimes_{\mathcal{H}_{par}} h.
     \]
 \end{proof}
 
 \begin{proposition} \label{p_PoHIsResolution}
     Let $P_\bullet \to \mathcal{B}$ be a projective resolution of right $\mathcal{H}_{par}$-modules of $\mathcal{B}$. Then, $P_\bullet \otimes_{\mathcal{H}_{par}} \mathcal{H}$ is a projective resolution of right $\mathcal{H}$-modules of $K$.
 \end{proposition}
 
 \begin{proof}
     By Lemma \ref{l_PHisHprojective} we know that $P_\bullet \otimes_{\mathcal{H}_{par}} \mathcal{H}$ is a complex of projective right $\mathcal{H}$-modules. Then, $P_\bullet \otimes_{\mathcal{H}_{par}} \mathcal{H}$ is a resolution of $K$ if, and only if, 
     \[
        H_n(P_\bullet \otimes_{\mathcal{H}_{par}} \mathcal{H}) = \left\{\begin{matrix}
            0 & \text{if }n \geq 1 \\ 
            K & \text{if }n = 0.
            \end{matrix}\right.
     \]
     Observe that by Lemma \ref{l_HoBFunctors}  and Lemma \ref{l_BprojectiveResolution} we have that
     \begin{align*}
         H_n(P_\bullet \otimes_{\mathcal{H}_{par}} \mathcal{H})
         &= H_n(P_\bullet \otimes_{\mathcal{B}} K) \\
         &= \operatorname{Tor}_n^{\mathcal{B}}(\mathcal{B},K) \cong   \left\{\begin{matrix}
             0 & \text{if }n \geq 1 \\ 
             K & \text{if }n = 0.
             \end{matrix}\right.
     \end{align*}
 \end{proof}
 
 \begin{remark}
     It follows from the proof of Proposition \ref{p_PoHIsResolution} that 
     \[
        H_n^{par}(\mathcal{H},\mathcal{H})= \operatorname{Tor}_n^{\mathcal{H}_{par}}(\mathcal{B}, \mathcal{H})= \operatorname{Tor}_n^{\mathcal{B}}(\mathcal{B},K) \cong   \left\{\begin{matrix}
             0 & \text{if }n \geq 1 \\ 
             K & \text{if }n = 0.
             \end{matrix}\right.
     \] 
 \end{remark}
 
 \begin{proposition} \label{p_GlobalCoincidesWithPartial}
     Let $X$ be a left (right) $\mathcal{H}$-module. Then,
     \[
         \operatorname{Tor}_\bullet^{\mathcal{H}_{par}}(\mathcal{B},X) =  \operatorname{Tor}_\bullet^{\mathcal{H}}(K,X) \text{ and } \operatorname{Ext}^\bullet_{\mathcal{H}_{par}}(\mathcal{B}, X) = \operatorname{Ext}^\bullet_{\mathcal{H}}(K, X).
     \]
 \end{proposition}
 
 \begin{proof}
     Let $P_\bullet$ be a projective resolution of $\mathcal{B}$ in \textbf{Mod}-$\mathcal{H}_{par}$. Then, if $X$ is a left $\mathcal{H}$-module we have
     \begin{align*}
         \operatorname{Tor}_n^{\mathcal{H}_{par}}(\mathcal{B},X)
         &= H_n(P_\bullet \otimes_{\mathcal{H}_{par}} X) \\
         &= H_n(P_\bullet \otimes_{\mathcal{H}_{par}} (\mathcal{H} \otimes_{\mathcal{H}} X)) \\
         &= H_n((P_\bullet \otimes_{\mathcal{H}_{par}} \mathcal{H}) \otimes_{\mathcal{H}} X) \\
     (\text{by Proposition } \ref{p_PoHIsResolution})    &= \operatorname{Tor}_\bullet^{\mathcal{H}}(K,X).
     \end{align*}
     In the dual settings, if $X$ is a right $\mathcal{H}$-module, we have
     \begin{align*}
         \operatorname{Ext}^n_{\mathcal{H}_{par}}(\mathcal{B},X)
         &= H_n(\operatorname{Hom}_{\mathcal{H}_{par}}(P_\bullet, X)) \\
         &= H_n\left( \operatorname{Hom}_{\mathcal{H}_{par}}(P_\bullet, \operatorname{Hom}_{\mathcal{H}}(\mathcal{H},X)) \right) \\
         &= H_n\left( \operatorname{Hom}_{\mathcal{H}_{par}}(P_\bullet, \operatorname{Hom}_{\mathcal{H}_{par}}(\mathcal{H},X)) \right) \\
         &= H_n\left( \operatorname{Hom}_{\mathcal{H}_{par}}(P_\bullet \otimes_{\mathcal{H}_{par}}\mathcal{H}, X) \right) \\
         (\text{by Proposition } \ref{p_PoHIsResolution})    &= \operatorname{Ext}^n_{\mathcal{H}}(K, X).
     \end{align*}
 \end{proof}
 
 Assume that we have a global action $\cdot:\mathcal{H} \otimes A \to A$. Then, any $A \# \mathcal{H}$-bimodule $M$ is an $\mathcal{H}$-module, consequently, $H_n(A,M)$ and $H^n(A,M)$ are $\mathcal{H}$-modules. Therefore, by Proposition \ref{p_GlobalCoincidesWithPartial} the spectral sequences of Theorems \ref{t_SSHH} and \ref{t_SSHC} take the global forms:
 \begin{corollary}\label{cor:globalHopf}
 Let $\mathcal{H}$ be a cocommutative Hopf algebra, which is projective as $K$-module, and let $\cdot:\mathcal{H} \otimes A \to A$  be a global action. Then there exist spectral sequences
    \[
        E^2_{p,q} = \operatorname{Tor}_p^{\mathcal{H}}(K,H_q(A,M) ) \Rightarrow H_{p+q}(A \#\mathcal{H}, M),
    \]
    and
    \[
            E_2^{p,q} = \operatorname{Ext}^p_{\mathcal{H}}(K,H^q(A,M) ) \Rightarrow H^{p+q}(A \#\mathcal{H}, M).
    \]
 \end{corollary}
 
%%%%%%%%%%%

\section{A projective resolution of \texorpdfstring{$\mathcal{B}$}{TEXT}}\label{sec:projResol}

Finally, we construct a projective resolution of $\mathcal{B}$. Define the map
     \begin{align*}
         \psi: \mathcal{H}_{par} &\to \mathcal{B} \\ 
                 x &\mapsto x \triangleright 1_{\mathcal{B}}.
     \end{align*}
Notice that $\psi$ is a morphism of $\mathcal{H}_{par}$-modules. Indeed, let $x,z \in \mathcal{H}_{par}$. Then,
     \[
         \psi(xz)=xz \triangleright 1_{\mathcal{B}} = x \triangleright (z \triangleright 1_{\mathcal{B}}) = x \triangleright \psi(z).
     \]
We define the modules of our resolution by
\begin{equation}\label{eq:proj}
    C_n'(\mathcal{H}):= \mathcal{H}_{par}^{\otimes _{\mathcal{B}}n+1},
\end{equation}
as the tensor product of $n+1$ copies of $\mathcal{H}_{par}$ over $\mathcal{B}$.
 Henceforth, for the sake of simplicity, since there is no ambiguity we will denote $C_n'(\mathcal{H})$ just by $C_n'$.
 
 \begin{proposition} \label{p_CnIsProjective}
     $C_n'$ is projective as left (right) $\mathcal{H}_{par}$-module.
 \end{proposition}
 
 \begin{proof}
     By induction, obviously $C_0'=\mathcal{H}_{par}$ is free as left (right) $\mathcal{H}_{par}$-module. Assume that $C_n'$ is projective as left (right) $\mathcal{H}_{par}$-module, by \cite[Proposition 3.9]{jacobson2012basic} we know that $C_{n+1}'$ is projective if, and only if, $\operatorname{Hom}_{\mathcal{H}_{par}}(C_{n+1}',-)$ is exact. First, we consider the right module case. Then, by \cite[Theorem 2.75]{RotmanAnInToHoAl} we have that
     \begin{align*}
         \operatorname{Hom}_{\mathcal{H}_{par}}(C_{n+1}',-) 
         &=  \operatorname{Hom}_{\mathcal{H}_{par}}(\mathcal{H}_{par} \otimes_{\mathcal{B}} C_{n}',-) \\ 
         &\cong \operatorname{Hom}_{\mathcal{B}}(\mathcal{H}_{par}, \operatorname{Hom}_{\mathcal{H}_{par}}(C_{n}',-)).
     \end{align*}
     Thus, $\operatorname{Hom}_{\mathcal{H}_{par}}(C_{n+1}',-)$ is exact since $C_{n}'$ is projective as a right $\mathcal{H}_{par}$-module by hypothesis and $\mathcal{H}_{par}$ is projective as a right $\mathcal{B}$-module. Analogously for the left module case by \cite[Theorem 2.76]{RotmanAnInToHoAl} we have that
     \begin{align*}
         \operatorname{Hom}_{\mathcal{H}_{par}}(C_{n+1}',-) 
         &=  \operatorname{Hom}_{\mathcal{H}_{par}}(C_{n}' \otimes_{\mathcal{B}} \mathcal{H}_{par} ,-) \\ 
         &\cong \operatorname{Hom}_{\mathcal{B}}(\mathcal{H}_{par}, \operatorname{Hom}_{\mathcal{H}_{par}}(C_{n}',-)),
     \end{align*}
     from which we conclude that $C_{n+1}'$ also is projective as a left $\mathcal{H}_{par}$-module.
 \end{proof}
 
 Observe that product in $\mathcal{H}_{par}$ is a well-defined morphism of $\mathcal{H}_{par}$-bimodules.
 \begin{align*}
     \mu: \mathcal{H}_{par} \otimes_{\mathcal{B}} \mathcal{H}_{par}  &\to \mathcal{H}_{par} \\ 
                 x \otimes_{\mathcal{B}} y                           &\mapsto xy.
 \end{align*}
 Using this map we are able to define the following face maps of left $\mathcal{H}_{par}$-modules for $n \geq 1$. If $0 \leq i \leq n-1$ we set $d_i: C_n' \to C_{n-1}'$ such that
 \[
     d_i:= id_{\mathcal{H}_{par}} \otimes_{\mathcal{B}} \dots \underset{i\text{-position}}{\underbrace{\otimes_{\mathcal{B}} \,  \mu  \, \otimes_{\mathcal{B}}}} \ldots \otimes_{\mathcal{B}} id_{\mathcal{H}_{par}}, \, \, 
 \]
 In particular, for a basic element of $C_n'$, this formula is determined by
 \begin{equation}
    d_i(x_0 \otimes_{\mathcal{B}} \ldots \otimes_{\mathcal{B}} x_n)= x_0 \otimes_{\mathcal{B}} \ldots \otimes_{\mathcal{B}} x_i x_{i+1} \otimes_{\mathcal{B}} \ldots \otimes_{\mathcal{B}} x_n.
 \end{equation}
 For $i = n$ we define
 \[
     d_n:= id_{\mathcal{H}_{par}} \otimes_{\mathcal{B}} \ldots \otimes_{\mathcal{B}} \mu(id_{\mathcal{H}_{par}} \otimes_{\mathcal{B}} \psi).
 \]
 Thus, for the basic elements of $C_n'$, we obtain
 \begin{equation}
     d_n(x_0 \otimes_{\mathcal{B}} \ldots \otimes_{\mathcal{B}} x_n) = x_0 \otimes_{\mathcal{B}} \ldots \otimes_{\mathcal{B}} x_{n-1} \psi(x_n)
 \end{equation}
 Analogously, we define the maps $s_i:C_n' \to C_{n+1}'$, $0 \leq i \leq n$ and $n \geq 0$, by
 
 \begin{equation}
     s_i(x_0 \otimes_{\mathcal{B}} \ldots \otimes_{\mathcal{B}} x_n) := x_0 \otimes_{\mathcal{B}} \ldots \otimes_{\mathcal{B}} x_i \otimes_{\mathcal{B}} 1_{\mathcal{H}_{par}} \otimes_{\mathcal{B}} x_{n+1} \otimes_{\mathcal{B}} \ldots \otimes_{\mathcal{B}} x_n.
 \end{equation}
 
 \begin{proposition}
     $(C_\bullet', d_i, s_i)$ is a simplicial module.
 \end{proposition}
 
 \begin{proof}
     Direct computations.
 \end{proof}
 
 Consequently, $(C_\bullet',\partial_\bullet)$ is a complex where
 \[
     \partial_n := \sum_{i=0}^n (-1)^i d_i.
 \]
 \begin{proposition}\label{prop:projResol}
     $C_\bullet' \overset{\psi}{\to} \mathcal{B}$ is a projective resolution of $\mathcal{B}$ in $\mathcal{H}_{par}$-mod.
 \end{proposition}
 
 \begin{proof}
     For $n \geq 0$ consider the map $s: C_n' \to C_{n+1}'$ such that
     \begin{equation}
         s(x_0 \otimes_{\mathcal{B}} \ldots \otimes_{\mathcal{B}} x_n) := 1_{\mathcal{H}_{par}} \otimes_{\mathcal{B}} x_0 \otimes_{\mathcal{B}} \ldots \otimes_{\mathcal{B}} x_n.
     \end{equation}
     Observe that such a map satisfies 
     \begin{enumerate}[(a)]
         \item $d_0 s= id_{\mathcal{C}_n'}$,
         \item $d_i s= sd_{i-1} $, for all $i \geq 1$.
     \end{enumerate}
     Thus, 
     \begin{align*}
         \partial_{n+1} s
         &= \sum_{i=0}^{n+1} (-1)^i d_i s
         = id_{C_n'} + \sum_{i=1}^{n+1} (-1)^i s d_{i-1}  \\ 
         &= id_{C_n'} - \sum_{i=0}^{n} (-1)^i s d_{i} 
         = id_{C_n'} - s \partial_{n}.
     \end{align*}
     Therefore, 
     \[
         s \partial_{n} + \partial_{n+1} s = id_{C_n'}, \, \forall n \geq 1.
     \]
     Whence $H_n(C_\bullet',\partial_\bullet)=0$ for all $n \geq 1$. Now consider map $\partial_1$. Explicitly this map takes the form
     \[
         \partial_1(x_0 \otimes_{\mathcal{B}} x_1) = x_0x_1 - x_0 \psi(x_1).
     \]
     Thus,
     \begin{align*}
         \psi \partial_0(x_0 \otimes_{\mathcal{B}} x_1) 
         &= \psi(x_0x_1) - \psi(x_0 \psi(x_1)) \\
         &= x_0x_1 \triangleright 1_{\mathcal{B}} - x_0 \psi(x_1) \triangleright 1_{\mathcal{B}} \\ 
         &= x_0 \triangleright (x_1 \triangleright 1_{\mathcal{B}}) - x_0 \triangleright (\psi(x_1) \triangleright 1_{\mathcal{B}}) \\
         (\text{by Lemma }\ref{l_BactionB})&= x_0 \triangleright \psi(x_1) - x_0 \triangleright \psi(x_1) = 0.
     \end{align*}
     Then, $C_\bullet' \overset{\psi}{\to} \mathcal{B}$ is a complex. Now, let $z$ be in the kernel of $\psi$, then
     \[
         \partial_1 s (z) 
         = \partial_1(1_{\mathcal{H}_{par}} \otimes_{\mathcal{B}} z) 
         = z - \psi(z) = z.
     \]
     Whence, $z \in \operatorname{im} \partial_1$. Thus, $\ker \psi = \operatorname{im} \partial_1$. Thus, $C_\bullet' \overset{\psi}{\to} \mathcal{B}$ is an exact sequence. Finally, by Proposition $\ref{p_CnIsProjective}$ each $C_n'$ is projective.
 \end{proof}

\section*{Acknowledgments}

The first named author was partially supported by 
Funda\c c\~ao de Amparo \`a Pesquisa do Estado de S\~ao Paulo (Fapesp), process n°:  2020/16594-0, and by  Conselho Nacional de Desenvolvimento Cient\'{\i}fico e Tecnol{\'o}gico (CNPq), process n°: 312683/2021-9. The second named author was supported by Fapesp, process n°: 2022/12963-7.

\bibliographystyle{abbrv}
 \bibliography{azu}

\end{document}